\documentclass[a4paper]{amsart}

\usepackage[utf8]{inputenc}
\usepackage{amssymb}
\usepackage[all]{xy}
\usepackage{nicefrac,mathtools,enumitem}
\usepackage{microtype}
\usepackage[pdftitle={Non-Hausdorff Symmetries of C*-algebras},
  pdfauthor={Alcides Buss, Ralf Meyer, Chenchang Zhu},
  pdfsubject={Mathematics}
]{hyperref}
\usepackage[lite]{amsrefs}
\newcommand*{\MRref}[2]{ \href{http://www.ams.org/mathscinet-getitem?mr=#1}{MR #1}}
\newcommand*{\arxiv}[1]{\href{http://www.arxiv.org/abs/#1}{arXiv: #1}}

\numberwithin{equation}{section}
\theoremstyle{plain}
\newtheorem{theorem}[equation]{Theorem}
\newtheorem{lemma}[equation]{Lemma}
\newtheorem{proposition}[equation]{Proposition}
\newtheorem{corollary}[equation]{Corollary}
\theoremstyle{definition}
\newtheorem{definition}[equation]{Definition}
\theoremstyle{remark}
\newtheorem{example}[equation]{Example}

\DeclareMathOperator{\coker}{coker}
\DeclareMathOperator{\Aut}{Aut}
\DeclareMathOperator{\Ad}{Ad}

\newcommand*{\nb}{\nobreakdash}
\newcommand*{\Star}{$^*$\nobreakdash-}

\newcommand*{\source}{\textup s}
\newcommand*{\range}{\textup r}

\newcommand*{\C}{\mathbb C}
\newcommand*{\Z}{\mathbb Z}
\newcommand*{\R}{\mathbb R}

\newcommand*{\Torus}{\mathbb T}
\newcommand*{\Comp}{\mathbb K}

\newcommand*{\Cst}{\textup C^*}
\newcommand*{\Cont}{\textup C}
\newcommand*{\Contc}{\textup C_\textup c}

\newcommand*{\Mult}{\mathcal M}
\newcommand*{\Bisec}{\mathcal S}
\newcommand*{\Rotc}[1]{\Cst(\Torus_#1)}

\newcommand*{\dd}{\textup d}
\newcommand*{\ima}{\textup i}
\newcommand*{\Eul}{\textup e}
\newcommand*{\Id}{\textup{Id}}

\newcommand*{\E}{\mathcal E} 
\newcommand*{\U}{\mathcal U} 
\newcommand*{\cm}{\mathcal{C}}
\newcommand*{\acm}{c}
\newcommand*{\tcm}{\partial}

\newcommand*{\defeq}{\mathrel{\vcentcolon=}}
\newcommand*{\into}{\rightarrowtail}
\newcommand*{\onto}{\twoheadrightarrow}

\newcommand*{\inv}{^{-1}}

\newcommand*{\abs}[1]{\lvert#1\rvert}
\newcommand*{\norm}[1]{\lVert#1\rVert}
\newcommand*{\cl}[1]{\overline{#1}}
\newcommand*{\braket}[2]{\left\langle#1\!\mid\!#2\right\rangle}

\begin{document}
\title{Non-Hausdorff Symmetries of $\Cst$-algebras}

\author{Alcides Buss}
\email{alcides@mtm.ufsc.br}
\address{Departamento de Matem\'atica\\
 Universidade Federal de Santa Catarina\\
 88.040-900 Florian\'opolis-SC\\
 Brazil}

\author{Ralf Meyer}
\email{rameyer@uni-math.gwdg.de}
\address{Mathematisches Institut and Courant Research Centre ``Higher Order Structures''\\
 Georg-August-Universit\"at G\"ottingen\\
 Bunsenstra\ss e 3--5\\
 37073 G\"ottingen\\
 Germany}

\author{Chenchang Zhu}
\email{zhu@uni-math.gwdg.de}
\address{Courant Research Centre ``Higher Order Structures''\\
 Georg-August-Universit\"at G\"ottingen\\
 Bunsenstra\ss e 3--5\\
 37073 G\"ottingen\\
 Germany}

\begin{abstract}
  Symmetry groups or groupoids of \(\Cst\)\nb-algebras associated to non-Hausdorff spaces are often non-Hausdorff as well.  We describe such symmetries using crossed modules of groupoids.  We define actions of crossed modules on \(\Cst\)\nb-algebras and crossed products for such actions, and justify these definitions with some basic general results and examples.
\end{abstract}
\subjclass[2000]{46L55, 18D05}
\keywords{\(\Cst\)\nb-algebra, crossed module, higher category theory, Green twisted action, locally Hausdorff groupoid, crossed product, Morita equivalence}

\thanks{Supported by the German Research Foundation (Deutsche Forschungsgemeinschaft (DFG)) through the Institutional Strategy of the University of G\"ottingen.}
\maketitle

\section{Introduction}
\label{sec:introduction}

Non-commutative geometry describes quotient spaces by non-commutative algebras, which here means \(\Cst\)\nb-algebras.  But symmetry groups of quotient spaces may be non-Hausdorff quotients as well.  Here we describe such non-Hausdorff symmetry groups and groupoids using crossed modules.  We define what it means for them to act on \(\Cst\)\nb-algebras, and we define crossed products for such actions.  Our definitions are motivated by higher category theory.  More precisely, they are special cases of general constructions with strict \(2\)\nb-categories.  To make this article easier to read for operator algebraists, we do not explicitly mention the general theory here, although it guides our definitions.

One motivating example is the gauge action on the rotation algebra~\(\Rotc\vartheta\) for an irrational number~\(\vartheta\).  This is the universal \(\Cst\)\nb-algebra generated by two unitaries \(U\) and~\(V\) that satisfy the commutation relation \(UV=\lambda VU\) with \(\lambda\defeq \exp(2\pi\ima\vartheta)\).  We may also view~\(\Rotc\vartheta\) as the crossed product \(C(\Torus)\rtimes_\lambda\Z\), where~\(\Z\) acts on~\(\Torus\) by \(n\cdot z\defeq \lambda^n\cdot z\).  Thus~\(\Rotc\vartheta\) describes the non-Hausdorff quotient space~\(\Torus/\lambda^\Z\), which is, in fact, a group.  We expect this group to act on itself by translations.  Thus we expect~\(\Rotc\vartheta\) to admit~\(\Torus/\lambda^\Z\) as a symmetry group.  The action of~\(\Torus\) is easy to describe: for \(z\in\Torus\), we have \(\alpha_z(U)= zU\) and \(\alpha_z(V)=V\).  This is one half of the familiar gauge action on~\(\Rotc\vartheta\).  The restriction of this action to the dense subgroup~\(\lambda^\Z\) is non-trivial, but inner: \(\alpha_\lambda(a) = V^*aV\) for all \(a\in\Rotc\vartheta\).

The above example is rather special because the group~\(\Torus\) is Abelian.  We generalise the above situation by considering \emph{crossed modules} of topological groupoids; in the introduction, we only explain the definition of crossed modules of topological groups for simplicity.  Crossed modules of discrete groups were introduced in homotopy theory to classify \(2\)\nb-connected spaces up to homotopy equivalence.  They are equivalent to strict \(2\)\nb-groups, which are central objects of study in higher category theory (see~\cites{Baez:Introduction_n, Noohi:two-groupoids}).  The crossed modules of topological groupoids that we introduce below are equivalent to strict topological \(2\)\nb-groupoids.

A crossed module of topological groups consists of two topological groups \(G\) and~\(H\) with a continuous group homomorphism \(\tcm\colon H\to G\) and a continuous left action~\(\acm\) of~\(G\) on~\(H\) by automorphisms, which satisfy the compatibility conditions
\begin{equation}
  \label{eq:crossed_module}
  \tcm\bigl(\acm_g(h)\bigr) = g \tcm(h)g^{-1}
  \qquad\text{and}\qquad
  \acm_{\tcm(h)}(k) = hkh^{-1}
\end{equation}
for all \(g\in G\), \(h,k\in H\).  For instance, \(H\) may be a closed normal subgroup in~\(G\), \(\tcm\colon H\to G\) the embedding, and~\(\acm\) the usual conjugation action \(\acm_g(h)\defeq ghg^{-1}\).  The crossed module \((G,H,\acm,\tcm)\) is a model for the possibly non-Hausdorff quotient group \(G/\tcm(H)\) in the same way that a locally compact groupoid is a model for its orbit space.

Actions of crossed modules on \(\Cst\)\nb-algebras are defined by copying the definition of twisted actions in the sense of Philip Green~\cite{Green:Local_twisted}.  Let \(\cm\defeq (G,H,\tcm,\acm)\) be a crossed module and let~\(A\) be a \(\Cst\)\nb-algebra.  An action of~\(\cm\) on~\(A\) consists of a strongly continuous action~\(\alpha\) of~\(G\) on~\(A\) and a strictly continuous homomorphism~\(u\) from~\(H\) to~\(\U\Mult(A)\), the unitary group of the multiplier algebra of~\(A\); the pair \((\alpha,u)\) is required to satisfy the two compatibility conditions
\begin{equation}
  \label{eq:action_crossed_module}
  \alpha_{\tcm(h)}(a) = u_hau_h^*
  \qquad\text{and}\qquad
  \alpha_g(u_h) = u_{\acm_g(h)}
\end{equation}
for all \(h\in H\), \(g\in G\), \(a\in A\).  The second condition uses the canonical extension of~\(\alpha_g\) to multipliers.

An interesting class of examples of such actions comes from crossed products.  Let~\(B\) be a \(\Cst\)\nb-algebra with a continuous action~\(\beta\) of~\(G\).  Then the crossed product \(A\defeq B\rtimes_{\beta\circ\tcm} H\) carries a canonical action of \((G,H,\tcm,\acm)\).  Let \(i_B\colon B\to\Mult(A)\) and \(i_H\colon H\to \U\Mult(A)\) be the canonical maps.  Define an action \(\alpha\colon G\to\Aut(A)\) by
\[
\alpha_g\bigl(i_B(b)\bigr) \defeq i_B\bigl(\beta_g(b)\bigr)
\qquad\text{and}\qquad
\alpha_g\bigl(i_H(h)\bigr) \defeq i_H\bigl(\acm_g(h)\bigr)
\]
for all \(g\in G\), \( b\in B\), \(h\in H\).  The pair \((\alpha,i_H)\) is an action of the crossed module \((G,H,\tcm,\acm)\) on~\(A\).

Let~\(D\) be another \(\Cst\)\nb-algebra.  A \emph{covariant representation} in \(\Mult(D)\) of an action \((\alpha,u)\) of a crossed module \((G,H,\tcm,\acm)\) on a \(\Cst\)\nb-algebra~\(A\) is a pair \((\pi,V)\) consisting of a non-degenerate \Star{}representation \(\pi\colon A\to\Mult(D)\) and a continuous group representation \(V\colon G\to\U\Mult(D)\), subject to the two compatibility conditions
\begin{equation}
  \label{eq:covariant_representation}
  \pi\bigl(\alpha_g(a)\bigr) = V_g\pi(a)V_g^*
  \qquad\text{and}\qquad
  \pi(u_h)=V_{\tcm(h)}
\end{equation}
for all \(g\in G\), \(a\in A\), \(h\in H\).  The \emph{crossed product} \(\Cst\)\nb-algebra \(A\rtimes_{(\alpha,u)}(G,H)\) is defined as the universal \(\Cst\)\nb-algebra for such covariant representations.  That is, non-degenerate \Star{}homomorphisms from \(A\rtimes_{(\alpha,u)}(G,H)\) to \(\Mult(D)\) correspond bijectively to covariant representations of the system in \(\Mult(D)\).  This universal property determines the crossed product uniquely.  We may construct it as a quotient of the usual crossed product \(A\rtimes_\alpha G\) because covariant representations of \((G,H,A,\alpha,u)\) are covariant representations of \((G,A,\alpha)\) that satisfy an extra condition involving~\(H\).

For the trivial action \(\alpha_g=\Id_\C\), \(u_h=1\) on~\(\C\), this yields the \(\Cst\)\nb-algebra of a crossed module.  Since covariant representations of this trivial system correspond bijectively to continuous representations of the quotient group \(G/\tcm(H)\), we get the somewhat disappointing result that \(\Cst(G,H)\) is the group \(\Cst\)\nb-algebra of the Hausdorff quotient \(G/\overline{\tcm(H)}\).  For the crossed module \(\Z\to\Torus\) that acts on~\(\Rotc\vartheta\), we simply get~\(\C\) because~\(\lambda^\Z\) is dense in~\(\Torus\).  Nevertheless, the above definition of the crossed product seems to be the correct one because it has several expected properties that we explain in the following paragraphs.

Let~\(B\) be a \(\Cst\)\nb-algebra with a continuous action~\(\beta\) of~\(G\).  We have already defined an action \((\alpha,i_H)\) of \((G,H,\tcm,\acm)\) on the crossed product \(B\rtimes_{\beta\circ\tcm}H\).  The crossed product for this action is
\begin{equation}
  \label{eq:crossed_product_stages}
  (B\rtimes_{\beta\circ\tcm}H)\rtimes_{\alpha,i_H}(G,H)\cong B\rtimes_\beta G,
\end{equation}
as expected.  For example, this applies to the canonical action of \((\Torus,\Z)\) on the rotation algebra~\(\Rotc\vartheta\):
\[
\Rotc\vartheta\rtimes(\Torus,\Z) \cong C(\Torus)\rtimes\Torus \cong \Comp\bigl(L^2(\Torus)\bigr).
\]

We do not yet have an intrinsic definition of free and proper actions of \(2\)\nb-groups, but the action of \((G,H)\) on \(\Cont_0(G,D)\rtimes H\) induced by the free and proper translation action of~\(G\) on itself should be an example.  The isomorphism above specialises to
\[
\bigl(\Cont_0(G,D)\rtimes H\bigr)\rtimes(G,H)\cong
\Cont_0(G,D)\rtimes G \cong \Comp\bigl( L^2(G) \bigr) \otimes D,
\]
which is the expected result.

The crossed product is clearly functorial.  If two actions of \((G,H)\) are equivariantly Morita equivalent in a suitable sense, then their crossed products are Morita--Rieffel equivalent \(\Cst\)\nb-algebras.

The crossed product functor is exact in the following sense: if \((G,H,\tcm,\acm)\) acts on~\(A\) and \(I\subseteq A\) is invariant under the \(G\)\nb-action, then \((G,H)\) acts on~\(I\) and on~\(A/I\), and the resulting crossed products form a \(\Cst\)\nb-algebra extension
\[
I\rtimes(G,H) \into A\rtimes(G,H) \onto A/I\rtimes(G,H).
\]

Although we have only discussed crossed modules of groups, an important motivation for this article is that several constructions of groupoids, such as holonomy groupoids of foliations or groupoids of germs of a pseudogroup of transformations, only produce locally Hausdorff groupoids in general: the space is non-Hausdorff, but covered by Hausdorff open subsets.  According to the principles of non-commutative geometry, this non-Hausdorff space of arrows should be viewed as the orbit space of another groupoid.  Higher category theory is designed to treat such situations.  In the more general setting of weak \(2\)\nb-groupoids, it is indeed possible to write every locally Hausdorff groupoid as the truncation of a Hausdorff topological weak \(2\)\nb-groupoid.  Since weak \(2\)\nb-groupoids are considerably more complicated than strict \(2\)\nb-groupoids --~which correspond to crossed modules of groupoids~-- we only treat some rather special locally Hausdorff groupoids here.

Let~\(G\) be a Hausdorff \'etale groupoid and let \(H\subseteq G\) be the interior of the set of loops in~\(G\) (arrows with same source and target).  Then the quotient \(G/H\) is a locally Hausdorff, \'etale groupoid.  The pair \((G,H)\) together with the embedding \(H\to G\) and the conjugation action of~\(G\) on~\(H\) is a crossed module of topological groupoids.  The \(\Cst\)\nb-algebra \(\Cst(G,H)\) in this case agrees with the \(\Cst\)\nb-algebra for locally Hausdorff, globally non-Hausdorff groupoids described in \cite{Connes:Survey_foliations}*{\S6}.

The advantage of the crossed module \((G,H)\) over the quotient \(G/H\) is that it has more actions on \(\Cst\)\nb-algebras.  In particular, it has many actions such as the one on \(\Cont_0(G)\rtimes H\) that deserve to be called proper.  We plan to use this to carry over the Dirac dual Dirac method to suitable locally Hausdorff groupoids.  But crossed modules are not general enough to cover all locally Hausdorff \'etale groupoids, so that we only discuss one rather simple special case here.

Throughout this article, we use the following conventions and terminology.  Locally compact means locally compact and Hausdorff.  Given a groupoid~\(G\), \(G_x\) denotes the source fibre at~\(x\) and~\(G^y\) denotes the range fibre at~\(y\), and~\(G_x^y\) denotes the space of arrows from~\(x\) to~\(y\).

\section{Crossed modules}
\label{sec:crossed_modules}

\begin{definition}
  \label{def:crossed_module}
  A \emph{crossed module} of groups is a quadruple \((G,H,\tcm,\acm)\) consisting of two groups \(G\) and~\(H\) and group homomorphisms \(\tcm\colon H\to G\) and \(\acm\colon G\to\Aut(H)\) that satisfy~\eqref{eq:crossed_module}, that is, \(\tcm\bigl(\acm_g(h)\bigr) = g \tcm(h)g^{-1}\) and \(\acm_{\tcm(h)}(k) = hkh^{-1}\).
\end{definition}

These conditions imply that \(\tcm(H)\) is a normal subgroup in~\(G\) and that \(\ker \tcm\) is a central subgroup of~\(H\).

Conversely, crossed modules with injective~\(\tcm\) are the same thing as normal subgroups \(H\triangleleft G\); we will focus on examples of this kind in our applications.  Crossed modules with surjective~\(\tcm\) are the same thing as central extensions \(H\onto G\).  In general, a crossed module is a central extension of the normal subgroup \(\tcm(H)\) of~\(G\) together with a lifting of the conjugation action of~\(G\) on~\(\tcm(H)\) to~\(H\) that extends the conjugation action of~\(\tcm(H)\) on the central extension~\(H\).

In the following, we often drop \(\acm\) and~\(\tcm\) from our notation and write \((G,H,\tcm)\) or just \((G,H)\) for a crossed module.

\begin{example}
  \label{exa:crossed_module_Cstar_Aut_unitary}
  Let~\(A\) be a \(\Cst\)\nb-algebra.  Let \(G\defeq\Aut(A)\) be the group of automorphisms of~\(A\) and let \(H\defeq \U\Mult(A)\) be the group of unitary multipliers of~\(A\).  Let \(\Ad_u\in \Aut(A)\) for \(u\in\U\Mult(A)\) be the inner automorphism \(a\mapsto uau^*\), and let~\(\Aut(A)\) act on~\(\U\Mult(A)\) by extending automorphisms of~\(A\) to \(\Mult(A)\) and then restricting them to \(\U\Mult(A)\).  This satisfies the conditions in~\eqref{eq:crossed_module}, so that we get a crossed module
  \[
  \Aut_2(A) \defeq (\Aut(A),\U\Mult(A),\Ad).
  \]
  The map~\(\Ad\) is, in general, neither injective nor surjective: its kernel consists of the central unitary multipliers, its range is the normal subgroup of inner automorphisms in \(\Aut(A)\).
\end{example}

Crossed modules were introduced by J.\,H.\,C. Whitehead~\cite{Whitehead:Combinatorial_homotopy_II} as algebraic models of homotopy \(2\)\nb-types.  This generalises the well-known fact that a connected \emph{aspherical} CW-complex is determined uniquely up to homotopy by its fundamental group.  Any crossed module \(\cm=(G,H,\tcm,\acm)\) has a classifying space~\(B\cm\), which is a certain connected CW-complex with \(\pi_1(B\cm) = \coker\tcm\) and \(\pi_2(B\cm)=\ker \tcm\), and \(\pi_n(B\cm)=0\) for \(n\ge3\).  For any connected CW-complex~\(X\), there is an associated crossed module~\(\cm\) and a map \(X\to B\cm\) that induces an isomorphism on \(\pi_1\) and~\(\pi_2\) and thus is a homotopy equivalence if \(\pi_n(X)=0\) for \(n\ge3\).  Thus any homotopy \(2\)\nb-type is represented by the classifying space of some crossed module (see also \cites{MacLane-Whitehead:3-type, Noohi:two-groupoids}).

We also consider crossed modules of \emph{topological} groups, where \(G\) and~\(H\) are topological groups and \(\tcm\colon H\to G\) and \(\acm\colon G\to\Aut(H)\) are continuous; continuity of~\(\acm\) means that the map \(G\times H\to H\), \((g,h)\mapsto \acm_g(h)\) is continuous.

\begin{example}
  \label{exa:crossed_module_Torus_Z}
  Let \(G=\Torus \defeq \{z\in\C^\times\mid \abs{z}=1\}\) and \(H=\Z\) with the trivial action of~\(G\).  Let \(\lambda=\exp(2\pi\ima\vartheta)\) for some irrational number~\(\vartheta\) and define \(\tcm\colon H\to G\) to be the embedding \(n\mapsto\lambda^n\).  This is a crossed module of topological groups.
\end{example}

\begin{example}
  \label{exa:Green_crossed_module}
  Let~\(G\) be a topological group and let~\(H\) be a closed normal subgroup with the subspace topology.  Let \(\tcm\colon H\to G\) be the embedding and let \(\acm_g(h)\defeq ghg^{-1}\).  This is a crossed module of topological groups; both \(\tcm\) and~\(\acm\) are automatically continuous.
\end{example}

\begin{example}
  \label{exa:B_group_crossed_module}
  Let~\(G\) be the trivial group~\(\{1\}\) and let~\(H\) be an abelian topological group.  The group~\(\{1\}\) acts trivially on~\(H\) and~\(\tcm\) is the unique map \(H\to\{1\}\).  Then \((\{1\},H)\) is a crossed module of topological groups.
\end{example}

\begin{example}
  \label{exa:Poincare_2-group}
  Let \(H\into E\onto G\) be a group extension with abelian~\(H\).  Then the action of~\(E\) on~\(H\) by conjugation descends to an action of~\(G\) on~\(H\) because~\(H\) is abelian.  Together with the trivial group homomorphism \(\tcm\colon H\to G\), we get a crossed module \((G,H,\tcm,\acm)\).

  If \(G=\textup{SO}(3,1)\) is the Lorentz group, \(H=\R^4\) is Minkowski space, and~\(E\) is the Poincar\'e group, this yields the so-called \emph{Poincar\'e \(2\)\nb-group}~\cite{Crane-Sheppeard:Poincare}.
\end{example}

The relationship between a crossed module \((G,H,\tcm,\acm)\) and the quotient group \(G/\tcm(H)\) is similar to the relationship between the groupoid \(X\rtimes G\) and the orbit space \(X/G\) for a group action of a group~\(G\) on a space~\(X\).  In particular, if~\(\tcm\) is a homeomorphism onto a closed normal subgroup of~\(G\), so that the left translation action of~\(H\) on~\(G\) is free and proper, then the crossed module \((G,H,\tcm,\acm)\) is equivalent to the quotient group \(G/\tcm(H)\) in a suitable sense, which we do not discuss here because this becomes clearer in the setting of higher category theory.

\begin{definition}
  \label{def:crossed_module_groupoid}
  A \emph{crossed module of topological groupoids} is a quadruple \((G,H,\tcm,\acm)\), where~\(G\) is a topological groupoid, \(H\) is a bundle of topological groups over the object space~\(X\) of~\(G\), \(\tcm\colon H\to G\) is a continuous homomorphism, and \(\acm\colon G \to\Aut(H)\) is a continuous action, such that
  \begin{alignat*}{2}
    \tcm\bigl(\acm_g(h)\bigr) &= g \tcm(h) g^{-1}&\qquad&
    \text{for all \((g, h) \in G\times_{\source, X} H\),}\\
    \acm_{\tcm(h)}(k) &= h k h^{-1}&\qquad& \text{for all \((h, k)\in H\times_X H\).}
  \end{alignat*}
  The action~\(\acm\) means that~\(\acm_g\) for \(g\in G_x^y\) is a group isomorphism from~\(H_x\) to~\(H_y\), and the map \(G\times_X H\to H\), \((g,h)\mapsto \acm_g(h)\) is continuous.  The equations \(\tcm\bigl(\acm_g(h)\bigr) = g \tcm(h) g^{-1}\) and \(\acm_{\tcm(h)}(k) = h k h^{-1}\) only make sense if \(g\in G_x\), \(h,k\in H_x\) for the same \(x\in X\).
\end{definition}

\begin{example}
  \label{exa:crossed_module_groupoid_isotropy}
  Let~\(G\) be a (Hausdorff) \'etale groupoid.  Recall that any \(g\in G\) extends to a local bisection~\(\bar{g}\) between open neighbourhoods of \(\source(g)\) and \(\range(g)\).  Let \(H\subseteq G\) be the set of all \(g\in G\) for which~\(\bar{g}\) is the identity map on a neighbourhood of~\(\source(g)\).  This subset is open and is a bundle of groups over the object space of~\(G\).  We call~\(H\) the \emph{interior isotropy bundle} of~\(G\).  More precisely, \(H\) is the interior of the closed subset of all \(g\in G\) with \(\range(g)=\source(g)\).  We equip~\(H\) with the subspace topology from~\(G\) and let \(\tcm\colon H\to G\) be the embedding.  If \(g\in G\) and \(h\in H_x\) are composable, then we have \(ghg^{-1} \in H_{gx}\).  This defines a continuous action of~\(G\) on~\(H\).  The data above defines a crossed module of topological groupoids.  This crossed module is a replacement for the groupoid of germs \(G/H\).  When~\(G\) is Hausdorff, the crossed module is always a crossed module of Hausdorff groupoids even though the quotient~\(G/H\) is locally Hausdorff but not necessarily Hausdorff.
\end{example}

\section{Actions of crossed modules on \texorpdfstring{$\Cst$}{C*}-algebras}
\label{sec:actions_crossed}

Let \(\cm\defeq (G,H,\tcm,\acm)\) be a crossed module of topological groups and let~\(A\) be a \(\Cst\)\nb-algebra.  Recall that automorphisms and unitary multipliers of~\(A\) form a crossed module \((\Aut(A),\U\Mult(A),\Ad)\) (Example~\ref{exa:crossed_module_Cstar_Aut_unitary}); even more, this is a crossed module of topological groups.

\begin{definition}
  \label{def:crossed_module_group_action}
  A (continuous) \emph{action} of~\(\cm\) on~\(A\) is a morphism of topological crossed modules from~\(\cm\) to \((\Aut(A),\U\Mult(A),\Ad)\).  That is, it is a pair of continuous group homomorphisms \(\alpha\colon G\to\Aut(A)\), \(u\colon H\to\U\Mult(A)\) such that \(\alpha_{\tcm(h)} = \Ad(u_h)\) for all \(h\in H\) and \(\alpha_g(u_h) = u_{\acm_g(h)}\) for all \(g\in G\), \(h\in H\).
\end{definition}

\begin{example}
  \label{exa:action_Green}
  Let~\(G\) be a locally compact group and let~\(H\) be a closed normal subgroup of~\(G\).  Actions of the resulting crossed module \(H\to G\) are, by definition, the twisted covariant systems introduced by Philip Green in~\cite{Green:Local_twisted}.
\end{example}

\begin{example}
  \label{exa:action_on_crossed}
  Let \(\cm=(G,H,\tcm,\acm)\) be a crossed module of locally compact topological groups.  Let~\(B\) be a \(\Cst\)\nb-algebra with a continuous action~\(\beta\) of~\(G\).  We describe a canonical action of \((G,H,\tcm,\acm)\) on the crossed product \(A\defeq B\rtimes_{\beta\circ\tcm} H\) (see~\eqref{eq:crossed-product}, where we briefly recall the definition of the crossed product).

  Let \(i_B\colon B\to\Mult(A)\) and \(i_H\colon H\to \U\Mult(A)\) be the canonical maps.  The action \(\alpha\colon G\to\Aut(A)\) is defined by
  \[
  \alpha_g\bigl(i_B(b)\bigr) \defeq i_B\bigl(\beta_g(b)\bigr), \qquad \alpha_g\bigl(i_H(h)\bigr) \defeq i_H\bigl(\acm_g(h)\bigr)
  \]
  for all \(g\in G\), \( b\in B\), \(h\in H\).  Notice that~\(\alpha_g\) preserves the commutation relation \(i_H(h)i_B(b) = i_B\bigl(\beta_{\tcm(h)}(b)\bigr) i_H(h)\):
  \begin{multline*}
    \alpha_g\bigl(i_H(h)i_B(b)\bigr) = i_H\bigl(\acm_g(h)\bigr) i_B\bigl(\beta_g(b)\bigr) = i_B\bigl(\beta_{\tcm\circ\acm_g(h)}\beta_g(b)\bigr) i_H\bigl(\acm_g(h)\bigr) \\= i_B\bigl(\beta_{g\tcm(h)g^{-1}g}(b)\bigr) i_H\bigl(\acm_g(h)\bigr) = \alpha_g\Bigl(i_B\bigl(\beta_{\tcm(h)}(b)\bigr)\Bigr) \alpha_g\bigl(i_H(h)\bigr).
  \end{multline*}
  The pair \((\alpha,i_H)\) is an action of the crossed module \((G,H,\tcm,\acm)\) on~\(A\).
\end{example}

\begin{example}
  \label{exa:action_Torus_Z}
  Now we specialise Example~\ref{exa:action_on_crossed} to the crossed module \((\Torus,\Z)\) in Example~\ref{exa:crossed_module_Torus_Z} and the action of~\(\Torus\) on \(B\defeq \Cont(\Torus)\) induced by the translation action.  Then \(B\rtimes_{\beta\circ\tcm} \Z\) is the irrational rotation algebra~\(\Rotc\vartheta\) and the action of \((\Torus,\Z)\) on it is given by \(\alpha_z(U) = zU\), \(\alpha_z(V)=V\) for all \(z\in\Torus\) and \(u_n=V^{-n}\) for all \(n\in\Z\).
\end{example}

\begin{example}
  \label{exa:action_B_group}
  Let~\(H\) be a locally compact abelian group and let~\(\hat{H}\) be the Pontrjagin dual of~\(H\).  We claim that actions on \(\Cst\)\nb-algebras of the crossed module \((\{1\},H)\) introduced in Example~\ref{exa:B_group_crossed_module} are equivalent to \(\Cst\)\nb-algebras over~\(\hat{H}\), that is, to \(\Cont_0(\hat{H})\)\nb-\(\Cst\)-algebras.

  By definition, an action of \((\{1\},H)\) on a \(\Cst\)\nb-algebra~\(A\) is a continuous group homomorphism from~\(H\) to \(Z\Mult(A)\), the centre of the multiplier algebra of~\(A\).  This is equivalent to a non-degenerate \Star{}homomorphism from \(\Cont_0(\hat{H}) \cong \Cst(H)\) to \(Z\Mult(A)\), that is, to a structure of \(\Cont_0(\hat{H})\)-\(\Cst\)-algebra on~\(A\).
\end{example}

\begin{example}
  \label{exa:action_Ab_extension}
  More generally, consider the crossed module \(\cm=(G,H,1,\acm)\) associated to an abelian group extension \(H\into E\onto G\) (see Example~\ref{exa:Poincare_2-group}).  The same argument as in Example~\ref{exa:action_B_group} shows that actions of~\(\cm\) are equivalent to actions of the transformation groupoid \(\hat{H}\rtimes G\), where we view the dual group~\(\hat{H}\) as a locally compact space and equip it with the induced action of~\(G\).  Recall that an action of \(\hat{H}\rtimes G\) on~\(A\) is equivalent to an action of~\(G\) together with a \(G\)\nb-equivariant structure of \(\Cont_0(\hat{H})\)-\(\Cst\)-algebra.

  In particular, actions of the Poincar\'e \(2\)\nb-group correspond to \(\textup{SO}(3,1)\)-equivariant \(\Cst\)\nb-algebras over \(\widehat{\R^4}\cong\R^4\).
\end{example}

Actions of crossed modules of topological groupoids are defined similarly, but the continuity requirements are harder to write down.

\begin{definition}
  \label{def:action_discrete_groupoid-cm_Cstar}
  Let \(\cm=(G,H,\tcm,\acm)\) be a crossed module of groupoids with object space \(X\).  Disregarding the issue of continuity for a moment, an \emph{action of~\(\cm\) on \(\Cst\)\nb-algebras} consists of a family of \(\Cst\)\nb-algebras \((A_x)_{x\in X}\), \Star{}isomorphisms \(\alpha_g\colon A_{\source(g)}\to A_{\range(g)}\) for all \(g\in G\) (\(\source(g)\) is the source and \(\range(g)\) the target of~\(g\)), and unitary multipliers \(u_h\in\Mult(A_x)\) for \(h\in H_x\), that satisfy analogous conditions: \(\alpha_{\tcm(h)} = \Ad(u_h)\) for all \(h\in H\) and \(\alpha_g(u_h) = u_{\acm_g(h)}\) for all \(g\in G\), \(h\in H_x\) with \(\source(g)=x\).
\end{definition}

We may interpret the above as a morphism of crossed modules of groupoids from \((G,H,\tcm,\acm)\) to a groupoid version of Example~\ref{exa:crossed_module_Cstar_Aut_unitary}, where we replace~\(\Aut(A)\) by the groupoid of \Star{}isomorphisms between \(\Cst\)\nb-algebras and the group~\(\U\Mult(A)\) by the bundle of groups with fibre \(\U\Mult(A)\) at the \(\Cst\)\nb-algebra~\(A\).

For simplicity, we only define continuity for actions of crossed modules of \emph{locally compact} topological groupoids, that is, the object space~\(X\) of~\(G\), the morphism space of~\(G\), and the total space of the bundle~\(H\) are now assumed to be locally compact spaces.  Using pro-\(\Cst\)-algebras instead of \(\Cst\)\nb-algebras, we could also treat compactly generated spaces instead of locally compact spaces.

We require the \(\Cst\)\nb-algebras \((A_x)_{x\in X}\) to be the fibres of a \(\Cst\)\nb-algebra~\(A\) over~\(X\) (this \(\Cst\)\nb-algebra over~\(X\) is part of our data).  We require the \Star{}isomorphisms \((\alpha_g)_{g\in G}\) to be the fibres of an isomorphism \(\alpha\colon \source^*A \to \range^*A\) of \(\Cst\)\nb-algebras over~\(G\), where we use the source and range maps to pull back~\(A\) to \(\Cst\)\nb-algebras over~\(G\) with fibres \(A_{\source(g)}\) and~\(A_{\range(g)}\), respectively.  And we require the \((u_h)_{h\in H}\) to be the fibres of a unitary multiplier~\(u\) of the pullback~\(p^*A\) of~\(A\) to~\(H\) along the bundle projection \(p\colon H\to X\) that maps~\(H_x\) to~\(x\).  The existence of \(\alpha\) and~\(u\) expresses the continuity of \((\alpha_g)\) and~\((u_h)\).

Since the map \(A\to \prod_{x\in X} A_x\) is injective for any \(\Cst\)\nb-algebra over~\(X\), \(\alpha\) and~\(u\) are unique if they exist at all.  Hence we do not have to specify them as additional data.  But~\(A\) and the structure of \(\Cst\)\nb-algebra over~\(X\) on~\(A\) are not yet determined by the fibres \((A_x)_{x\in X}\), so that we must specify this as additional data.

With the above definition, Example~\ref{exa:action_on_crossed} extends to locally compact groupoids with Haar system.  Let \((G,H,\tcm,\acm)\) be a crossed module of locally Hausdorff locally compact groupoids with Haar systems (see~\cite{Paterson:Groupoids}) and let~\(B\) carry a continuous action~\(\beta\) of~\(G\).  In particular, \(B\) is a \(\Cst\)\nb-algebra over the object space~\(X\) of~\(G\).  So is the crossed product \(B\rtimes_{\beta|_H}H\) --~with fibres \(\bigl(B\rtimes_{\beta|_H}H\bigr)_x=B_x\rtimes_{\beta|_{H_x}}H_x\)~-- because~\(H\) is a bundle of groups over~\(X\).  Let \(\Contc(H,B)\) be the space of continuous, compactly supported sections of the pull-back of~\(B\) to~\(H\); this is a dense subalgebra of \(B\rtimes_{\beta|_H}H\).  The pull-backs \(\source^*\Contc(H,B)\) and \(\range^*\Contc(H,B)\) are the spaces of compactly supported sections of bundles over~\(G\) with fibres \(\Contc(H_{\source(g)},B_{\source(g)})\) and \(\Contc(H_{\range(g)},B_{\range(g)})\) at \(g\in G\), respectively.  For \(g\in G\), we define an automorphism \(\alpha\colon \source^*\Contc(H,B) \to \range^*\Contc(H,B)\) by \((\alpha f)(g,h) \defeq \beta_g\bigl(f(g,\acm_g^{-1}(h))\bigr)\) for all \(g\in G\), \(h\in H\).  The action~\(\alpha\) is continuous in the inductive limit topology and hence extends to a continuous action on the \(\Cst\)\nb-level.  The continuity of the map \(i_H\colon H\to \U\Mult(p^*A)\) is trivial to verify.  Thus~\(A\) carries a continuous action of the crossed module \((G,H,\tcm,\acm)\).

Let \(A\) and~\(B\) be \(\Cst\)\nb-algebras over~\(X\) and let \(A\otimes_X B\) denote the restriction to the diagonal \(X\subseteq X\times X\) of their minimal \(\Cst\)\nb-tensor product (the same assertions also hold for the maximal \(\Cst\)\nb-tensor product instead).  If \((\alpha,u)\) and \((\beta,v)\) are continuous actions of a crossed module of topological groupoids \((G,H,\tcm,\acm)\) on \(A\) and~\(B\), then \(A\otimes_X B\) carries a canonical action \((\alpha\otimes_X\beta,u\otimes_X v)\) of \((G,H,\tcm,\acm)\), called the \emph{diagonal action}.  It is defined by \((\alpha\otimes_X\beta)_g(a\otimes b) \defeq \alpha_g(a) \otimes \beta_g(b)\), \((u\otimes_X v)_h\cdot (a\otimes b) \defeq u_h\cdot a\otimes v_h\cdot b\) for all \(g\in G_x\), \(h\in H_x\), \(a\in A_x\), \(b\in B_x\).

The tensor product with the diagonal action defines a bifunctor on the category of \(\Cst\)\nb-algebras with \((G,H,\tcm,\acm)\)-action.  This bifunctor is associative and commutative, and the obvious action on~\(\Cont_0(X)\), where~\(X\) is the unit space of~\(G\) (defined by \(\alpha_gf(x)= f(g^{-1}x)\) for all \(g\in G\), \(x\in X\) and \(u_h=1\) for all \(h\in H\)) is a unit object.  Hence the category of \(\Cst\)\nb-algebras with \((G,H,\tcm,\acm)\)-action becomes a symmetric monoidal category.

\begin{example}
  \label{exa:tensor_B_group_action}
  Consider the crossed module \((\{1\},H)\) for a locally compact abelian group~\(H\).  We have seen in Example~\ref{exa:action_B_group} that actions of \((\{1\},H)\) correspond to \(\Cst\)\nb-algebras over the dual group~\(\hat{H}\).  The tensor product of actions of \((\{1\},H)\) corresponds to a tensor product for \(\Cst\)\nb-algebras over~\(\hat{H}\).  If \(A\) and~\(B\) are \(\Cst\)\nb-algebras over~\(\hat{H}\), then their tensor product \(A\otimes B\) in the usual sense is a \(\Cst\)\nb-algebra over \(\hat{H}\times\hat{H}\).  We use the map \(\Cont_0(\hat{H})\to\Cont_b(\hat{H}\times\hat{H})\) induced by the group structure of~\(\hat{H}\) to view \(A\otimes B\) as a \(\Cst\)\nb-algebra over~\(\hat{H}\) once again.  Thus the fibre of \(A\otimes B\) at~\(\xi\in\hat{H}\) is the space of \(\Cont_0\)\nb-sections of the field of \(\Cst\)\nb-algebras over~\(\hat{H}\) with fibre \(A_\eta\otimes B_{\eta^{-1}\xi}\) at \(\eta\in\hat{H}\).
\end{example}

\subsection{Actions of crossed modules on groupoids}
\label{sec:actions_on_groupoids}

We are going to define actions of crossed modules on groupoids and use them to induce actions on groupoid \(\Cst\)\nb-algebras.  As for the action on a groupoid, our definition is based on a crossed module associated to the groupoid:

\begin{example}
  \label{exa:crossed_module_groupoid}
  Let~\(K\) be a locally compact groupoid.  We define a crossed module \(\Aut_2(K) = (\Aut(K),\Bisec,\tcm,\acm)\) of topological groups that combines groupoid automorphisms and bisections (or inner automorphisms).  We let~\(\Aut(K)\) be the topological group of automorphisms of~\(K\) --~these are pairs of homeomorphisms \(\alpha^{(0)}\colon K^{(0)} \to K^{(0)}\) and \(\alpha^{(1)}\colon K^{(1)} \to K^{(1)}\) that intertwine the unit, range, source, and multiplication maps in~\(K\); we equip~\(\Aut(K)\) with the compact open topology.

  A global bisection of~\(K\) is a closed subset \(S\subseteq K^{(1)}\) such that both \(\range\) and~\(\source\) restrict to homeomorphisms \(S\to K^{(0)}\).  Equivalently, we may describe~\(S\) as the graph of a section \(h\colon K^{(0)}\to K^{(1)}\) of the source map, that is, \(\source\circ h = \Id_{K^{(0)}}\), letting \(h(x)\) be the unique element of~\(S\) with \(\source\bigl(h(x)\bigr) = x\).  Global bisections correspond to sections~\(h\) of~\(\source\) for which \(\range\circ h\colon K^{(0)}\to K^{(0)}\) is a homeomorphism.  We topologise the space of global bisections using the compact open topology for maps \(K^{(0)}\to K^{(1)}\).  There is a unit bisection \(K^{(0)}\subseteq K^{(1)}\) and if \(S\) and~\(T\) are global bisections, so are \(S^{-1}\) and \(S\cdot T\); when we pass to sections \(h_S\colon K^{(0)}\to K^{(1)}\) and \(h_T\colon K^{(0)}\to K^{(1)}\), we get \(h_{S^{-1}}(x) = h_S\bigl((\range h_S)^{-1}(x)\bigr)^{-1}\) and \(h_{ST}(x) = h_S\bigl(\range h_T(x)\bigr)\cdot h_T(x)\):
  \[
  x \xleftarrow{h_S\bigl((\range h_S)^{-1}(x)\bigr)} (\range h_S)^{-1}(x),\qquad x \xrightarrow{h_T(x)} \range h_T(x) \xrightarrow{h_S\bigl(\range h_T(x)\bigr)} \range h_S \bigl(\range h_T(x)\bigr).
  \]
  As a result, global bisections form a topological group~\(\Bisec(K)\).  Obviously, the automorphism group~\(\Aut(K)\) acts by group automorphisms on the group of bisections~\(\Bisec(K)\).

  A global bisection~\(S\) generates an automorphism of~\(K\) by \(\alpha^{(0)}(x) = \range h_S(x)\), \(\alpha^{(1)}(k) = h_S\bigl(\range(k)\bigr) \cdot k \cdot h_S\bigl(\source(k)\bigr)^{-1}\) for \(x\in K^{(0)}\), \(k\in K^{(1)}\).  The following is easy to check: this map is indeed an automorphism of~\(K\); the resulting map \(\tcm\colon \Bisec(K)\to \Aut(K)\) is a group homomorphism; together with the obvious action of~\(\Aut(K)\) on~\(\Bisec(K)\), this yields a crossed module of topological groups.

  To motivate the above construction, we may interpret groupoids as categories and automorphisms of groupoids as invertible functors.  Then a global bisection~\(h\) becomes a natural isomorphism from~\(g\) to \(\tcm(h)\circ g\) for any automorphism~\(g\) of~\(K\), and all natural isomorphisms between automorphisms of~\(K\) are of this form for some global bisection~\(h\).
\end{example}

\begin{example}
  \label{exa:Aut_of_group}
  If~\(K\) is a group, that is, \(K^{(0)}\) is a point, then~\(\Aut(K)\) is the group of group automorphisms of~\(K\) in the usual sense, \(\Bisec(K)=K\), and \(\tcm\colon K\to\Aut(K)\) maps \(k\in K\) to the associated inner automorphism of~\(K\).
\end{example}

\begin{example}
  \label{exa:Aut_space}
  If~\(K\) is a topological space viewed as a groupoid with only identity morphisms, then \(\Aut_2(K)\) is the topological group of homeomorphisms \(K\to K\); the group of bisections is trivial.
\end{example}

\emph{An action of a crossed module of topological groups} \((G,H,\tcm,\acm)\) on a locally compact groupoid~\(K\) is defined as a continuous morphism of crossed modules to the crossed module \(\Aut_2(K)=(\Aut(K),\Bisec(K))\) constructed in Example~\ref{exa:crossed_module_groupoid}.  That is, \(G\) acts on~\(K\) by automorphisms and~\(H\) by global bisections, satisfying the usual two compatibility conditions.

\emph{Actions of crossed modules of topological groupoids} are notationally more complicated.  Let \(\cm=(G,H,\tcm,\acm)\) be a crossed module of topological groupoids with object space~\(X\), let \(\pi\colon H\to X\) be the bundle projection and let~\(K\) be a topological groupoid with a continuous groupoid morphism \(\rho\colon (K^{(1)}\rightrightarrows K^{(0)}) \to (X \rightrightarrows X) \).  Roughly speaking, \(K\) is a continuous field of groupoids over~\(X\) whose fibre~\(K_x\) at \(x\in X\) is the restriction of~\(K\) to \(\rho^{-1}(x) \subseteq K^{(0)}\).  We require~\(G\) to act on this field by groupoid isomorphisms, that is, we are given groupoid isomorphisms \(\alpha_g\colon K_{\source(g)}\to K_{\range(g)}\) for all \(g\in G\).  These are continuous in the sense that they piece together to continuous maps \(G\times_{\source,\rho} K^{(j)} \to G\times_{\range,\rho} K^{(j)}\) for \(j=0,1\), that is, both on objects and arrows.  Furthermore, we are given global bisections \(\kappa_h\colon K_x^{(0)}\to K_x^{(1)}\) for all \(x\in X\), \(h\in H_x\).  These are continuous in the sense that they combine to a continuous map \(H\times_{\pi,\rho} K^{(0)} \to K^{(1)}\).  The compatibility conditions for a crossed module action require \(\alpha_{\tcm(h)}(k) = \kappa_h\bigl(\range(k)\bigr)\cdot k\cdot \kappa_h\bigl(\source(k)\bigr)^{-1}\) and \(\kappa_{\acm_g(h)}(y) = \alpha_g\bigl( \kappa_h(\alpha_g^{-1}(y))\bigr)\) for all \(g\in G\), \(h\in H_{\source(g)}\), \(y\in K_{\range(g)}^{(0)}\).

In other words, we form the transformation groupoid \(H\ltimes_\tcm G\) for the left action of~\(H\) on the space of arrows in~\(G\) by \(h\cdot g \defeq \tcm(h) g\); then an action of \((G, H, \tcm, \acm)\) on~\(K\) is a groupoid morphism \((H\ltimes_\tcm G) \times_X K \to K\) satisfying the associativity condition as for usual actions.

Actions on groupoids should play the same role for crossed modules as actions on spaces for groupoids.  In particular, the classifying space or the universal proper action of a crossed module should be such an action on a groupoid.

\begin{example}
  \label{exa:action_on_equivalence_relation}
  Consider the special case where the groupoid~\(K\) comes from an equivalence relation~\(\sim\) on~\(K^{(0)}\), that is, \(K_x^x=\{1_x\}\) for all \(x\in K^{(0)}\).  Here the structure above simplifies considerably because everything is already determined by the action of~\(G\) on the object space~\(K^{(0)}\) of~\(K\).  We disregard continuity, this has to be checked directly both on \(K^{(0)}\) and~\(K^{(1)}\).

  Let \(\rho\colon K^{(0)}\to X\) be a map to the common object space of \(G\) and~\(H\).  This descends to \(K^{(1)}\backslash K^{(0)}\) if and only if \(x\sim y\) implies \(\rho(x)=\rho(y)\).  If this is the case, we let~\(K_x\) be the restriction of~\(K\) to the pre-image of \(x\in X\).  A bijection \(\alpha\colon K_x^{(0)}\to K_y^{(0)}\) extends to a groupoid isomorphism \(K_x\to K_y\) if and only if \(\alpha(a)\sim\alpha(b)\) for all \(a, b\in K_x^{(0)}\) with \(a\sim b\), and this groupoid isomorphism is determined uniquely.  A bisection~\(h\) must associate to each \(a\in K_x^{(0)}\) an element \(h(a)\in K_x^{(0)}\) with \(a\sim h(a)\).  The action of~\(H\) is implemented by bisections if and only if \(\alpha_h(a)\sim a\) for all \(h\in H_x\), \(a\in K_x^{(0)}\), and once again the bisections are determined uniquely.

  Summing up, an action of a crossed module \((G,H,\tcm,\acm)\) on~\(K\) is determined by an action of~\(G\) on~\(K^{(0)}\) that lifts the action on~\(X\) and that satisfies \(\alpha_g(a)\sim\alpha_g(b)\) for all \(g\in G_x\), \(a,b\in K_x^{(0)}\) with \(a\sim b\) and \(\alpha_{\tcm(h)}(a)\sim a\) for all \(h\in H_x\), \(a\in K_x^{(0)}\).
\end{example}

\begin{lemma}
  \label{lem:induced_action_groupoid_Cstar}
  Let~\(\cm\) be a crossed module of locally compact groupoids and let~\(K\) be a locally compact groupoid with Haar system.  Then an action of~\(\cm\) on~\(K\) induces an action on the groupoid \(\Cst\)\nb-algebra \(\Cst(K)\).
\end{lemma}

The \(\Cst\)\nb-algebra \(\Cst(K)\) is defined in~\cite{Renault:Representations}.  We will briefly recall the definition of crossed products in Section~\ref{sec:covariant_rep_crossed}.  The following proof already uses the standard notation without further explanation.

\begin{proof}
  Let \(\cm=(G,H,\tcm,\acm)\) and let~\(X\) be the common object space of \(G\) and~\(H\).  We define an action of~\(\cm\) on the dense \Star{}subalgebra \(\Contc(K)\) of compactly supported continuous functions on the space of arrows on~\(K\).  This is a \(\Cont_0(X)\)-algebra via the anchor map \(\rho\colon K\to X\) that is part of the action of~\(\cm\).  Pulling it back to~\(G\) via the source and range maps, we get the spaces of compactly supported functions on \(G\times_{\source,\rho} K\) and \(G\times_{\range,\rho} K\).  The action of~\(G\) on~\(K\) provides an homeomorphism between these spaces, which defines an action of the groupoid~\(G\) on~\(\Contc(K)\) by \Star{}algebra automorphisms.  Being manifestly bounded for the \(I\)\nb-norm, this action extends to one on the \(\Cst\)\nb-completion~\(\Cst(K)\).  Any bisection \(h\colon K_x^{(0)}\to K_x^{(1)}\) defines a unitary multiplier~\(u_h\) of~\(\Contc(K_x)\) by \(u_h f(k) = f(h^{-1}(\cdots)\cdot k)\) and \(fu_h(k) = f\bigl(k \cdot h^{-1}(\cdots)\bigr)\) or, more precisely, \(u_h\cdot f(k) \defeq f\bigl(h\bigl((\range\circ h)^{-1}\circ \range (k)\bigr)^{-1}\cdot k\bigr)\) and \(f\cdot u_h(k) \defeq f\bigl(k\cdot h(\source k)^{-1}\bigr)\).  Letting~\(x\) vary, we get a unitary multiplier of the pullback of \(\Contc(K)\) from~\(X\) to~\(H\).  This remains a multiplier of the \(\Cst\)\nb-completions, and completes the action of~\(\cm\) on~\(\Cst(K)\).
\end{proof}

\begin{example}
  \label{exa:action_nc_torus_from_groupoid}
  Let \(\cm=(G,H,\tcm,\acm)\) be a crossed module of locally compact groupoids with common object space~\(X\).  Thinking of~\(\cm\) as representing a non-Hausdorff groupoid, the transformation groupoid \(H\ltimes_\partial G\) describes the non-Hausdorff space of arrows of~\(\cm\).  We expect, then, that~\(\cm\) acts on \(H\ltimes_\partial G\) by left translation.  And this should be the prototype of a transitive free and proper action.  This action of~\(\cm\) on \(H\ltimes_\partial G\) is defined as follows: let \(g\in G\) act on~\(G^{\range(g)}\) by left translation by~\(g\) and on \(H_{\source(g)}\times G^{\range(g)}\) by \(g\cdot (h,g') \defeq (\acm_g(h),gg')\); finally, map \(h\in H_x\) to the constant bisection \(G^x\ni g\mapsto (h,g) \in H_x\times G^x\).

  The induced action of~\(\cm\) on the groupoid \(\Cst\)\nb-algebra \(\Cst(H\ltimes_\partial G) \cong H\ltimes \Cont_0(G)\) is also a special case of the construction in Example~\ref{exa:action_on_crossed}: apply the latter to the action of~\(G\) on \(\Cont_0(G)\) by left translations.

  In particular, the construction above generalises the canonical action of the crossed module \((\Torus,\Z,\lambda)\) from Example~\ref{exa:crossed_module_Torus_Z} on the rotation algebra~\(\Rotc\theta\).
\end{example}

\section{Covariant representations and crossed products}
\label{sec:covariant_rep_crossed}

We first define covariant representations of crossed modules of topological groupoids.  For crossed modules of locally compact groups, we define the crossed product by a universal property with respect to such covariant representations.  Such a definition is inconvenient in the groupoid setting because the desintegration theory for groupoid representations is rather technical.  Instead, we reduce the construction of crossed products for crossed modules of groupoids to the already known crossed product for ordinary groupoids.  For crossed modules of groups, this construction produces a \(\Cst\)\nb-algebra with the correct universal property.

The following definition generalises one by Philip Green in~\cite{Green:Local_twisted}.

\begin{definition}
  \label{def:covariant_rep}
  Let \((G,H,\tcm,\acm)\) be a crossed module of topological groups, let \(A\) and~\(D\) be \(\Cst\)\nb-algebras, and let \((\alpha,u)\) be an action of~\((G, H, \tcm,\acm)\) on~\(A\).  A \emph{covariant representation} of \((G,H,\tcm,A,\alpha,u)\) in \(\Mult(D)\) is a pair \((\pi,V)\) consisting of a non-degenerate \Star{}homomorphism \(\pi\colon A\to \Mult(D)\) and a strictly continuous group homomorphism \(V\colon G\to\U\Mult(D)\) that satisfy
  \begin{alignat*}{2}
    V_g\pi(a)V_g^* &= \pi\bigl(\alpha_g(a)\bigr)
    &\qquad&\text{for all \(g\in G\), \(a\in A\);}\\
    V_{\tcm(h)} &= \pi(u_h) &\qquad&\text{for all \(h\in H\).}
  \end{alignat*}
\end{definition}

\begin{definition}
  \label{def:crossed_product}
  A covariant representation \((i_A,i_G)\) of \((G,H,\tcm,A,\alpha,u)\) in \(\Mult(D)\) is called \emph{universal} if any covariant representation of the same data on a \(\Cst\)\nb-algebra~\(E\) is of the form \((f\circ i_A,f\circ i_G)\) for a unique strictly continuous, unital \Star{}homomorphism \(f\colon \Mult(D)\to\Mult(E)\); equivalently, \(f\) is the strictly continuous extension of a non-degenerate \Star{}homomorphism \(D\to\Mult(E)\).  If \((i_A,i_G)\) is universal, we call~\(D\) the \emph{crossed product} of \((G,H,\tcm,A,\alpha,u)\) and denote it by \(A\rtimes_{(\alpha,u)}(G,H,\tcm,\acm)\) or, more briefly, by \(A\rtimes(G,H)\).
\end{definition}

The crossed product is unique up to \Star{}isomorphism.  More precisely, if two covariant representations \((i_A,i_G)\) and \((i'_A,i'_G)\) in \(\Mult(D)\) and~\(\Mult(D')\) are universal, then there is a unique strictly continuous \Star{}isomorphism \(\Mult(D)\cong\Mult(D')\) that intertwines \((i_A,i_G)\) and \((i'_A,i'_G)\).  Such an isomorphism restricts to an isomorphism \(D\cong D'\).

While Definition~\ref{def:crossed_product} makes sense for any topological group~\(G\), the existence and the following construction of a universal object require~\(G\) to be locally compact.

Following~\cite{Green:Local_twisted}, our starting point is the crossed product algebra \(A\rtimes_\alpha G\), which satisfies a similar universal property for covariant representations of \((G,A,\alpha)\), with respect to a canonical \Star{}representation \(i_A\colon A\to \Mult(A\rtimes_\alpha G)\) and a canonical unitary representation \(i_G\colon G\to\Mult(A\rtimes_\alpha G)\).  We will construct the crossed product as a quotient of \(A\rtimes_\alpha G\).  This construction still works for groupoids, where the universal property in terms of covariant representations is more complicated because of the difficult measure theory involved in the integration and desintegration of groupoid representations (see~\cite{Renault:Representations}).

Therefore, we now consider the more general setting of a crossed module of topological groupoids \(\cm = (G,H,\tcm,\acm)\) (see Definition~\ref{def:crossed_module_groupoid}).  We assume that \(H\) and~\(G\) are locally compact groupoids with Haar systems, and we denote the common object space of \(G\) and~\(H\) by~\(X\).  For simplicity, we also assume \(G\) and~\(H\) to be Hausdorff, although this assumption is probably unnecessary: with some effort it should be possible to generalise the same construction to locally Hausdorff groupoids.

Let \((A,\alpha,u)\) be a continuous action of~\(\cm\).  This means that~\(A\) is a \(\Cst\)\nb-algebra over~\(X\) (a \(\Cont_0(X)\)-\(\Cst\)-algebra), \(\alpha\) is a continuous action of~\(G\) on~\(A\), and~\(u\) is a continuous homomorphism between the group bundles~\(H\) and \(\U\Mult(A_x)_{x\in X}\) over~\(X\).  We want to construct a crossed product \(\Cst\)\nb-algebra \(A \rtimes_{\alpha,u} (G,H)\), using ordinary crossed products for groupoid actions.

First we define a crossed product groupoid \(H\rtimes_\acm G\) with the same object space~\(X\) as \(G\) and~\(H\).  Its space of arrows is the fibred product \(H\times_X G\), the multiplication is defined by \((h_1,g_1)\cdot (h_2,g_2) \defeq (h_1\acm_{g_1}(h_2),g_1g_2)\) for all composable pairs.  This defines a locally compact groupoid, which has a Haar system: take the product of the Haar systems on \(H\) and~\(G\).  The first condition for a crossed module of groupoids asserts that the map \(H\rtimes_\acm G \to G\), \((h,g)\mapsto \tcm(h)g\) is a continuous groupoid homomorphism.  Hence the action~\(\alpha\) of~\(G\) yields a continuous action~\(\bar\alpha\) of \(H\rtimes_\acm G\) by \(\bar\alpha_{(h,g)} \defeq \alpha_{\tcm(h)g}\).

Since both \(G\) and~\(H\rtimes_\acm G\) are locally compact groupoids with Haar systems, the definitions in~\cite{Renault:Representations} provide crossed product \(\Cst\)\nb-algebras \(A\rtimes_\alpha G\) and \(A \rtimes_{\bar\alpha} (H\rtimes_\acm G)\), whose definitions we now recall.  The crossed product \(\Cst\)\nb-algebra \(A\rtimes_\alpha G\) is defined by completing a dense \Star{}subalgebra \(\Cont_c(G,A)\) with respect to a suitable maximal \(\Cst\)\nb-norm. Here \(\Cont_c(G,A)\) denotes the space of continuous, compactly supported sections of the pull-back bundle \(\range^*(A)\), which becomes a \Star{}algebra by
\begin{equation}
  \label{eq:crossed-product}
  (f_1*f_2)(g) \defeq
  \int_{G^{\range(g)}} f_1(h) \alpha_h\bigl(f_2(h^{-1}g)\bigr) \,\dd\lambda^{\range(g)}(h),\quad
  f^*(g) \defeq \alpha_g\bigl(f(g\inv)^*\bigr).
\end{equation}
Here \((\lambda^x)_{x\in X}\) denotes the left invariant Haar measure on~\(G\).  The \(I\)\nb-norm on \(\Cont_c(G,A)\) is defined by
\[
\|f\|_I\defeq \max \left\{ \sup_{x\in X} \int_{G^x} \norm{f(g)} \,\dd{\lambda^x}(g),\quad \sup_{x\in X} \int_{G^x} \norm{f^*(g)} \,\dd{\lambda^x}(g) \right\}.
\]
A \Star{}representation of \(\Cont_c(G,A)\) on Hilbert space is called \emph{bounded} if it is contractive with respect to this norm.  Each \Star{}representation of \(\Cont_c(G,A)\) induces a \(\Cst\)\nb-seminorm on \(\Cont_c(G,A)\).  The supremum of these norms for all contractive representations is a well-defined \(\Cst\)\nb-norm on \(\Cont_c(G,A)\), and \(A\rtimes_\alpha G\) is the resulting \(\Cst\)\nb-completion.

Now we temporarily specialise once again to the group case to motivate the crucial ingredients of our construction.  We define two unitary representations of~\(H\) in \(\Mult(A\rtimes_\alpha G)\) by \(\rho_h\defeq i_G\bigl(\tcm(h)\bigr)\) and \(\sigma_h\defeq i_A(u_h)\) for \(h\in H\).  These representations are relevant because a covariant representation of \((G,H,\tcm,A,\alpha,u)\) is nothing but a covariant representation \((\pi,V)\) of \((G,A,\alpha)\) that satisfies the extra condition \(\pi\circ u = V\circ\tcm\).  Equivalently, the strictly continuous extension of the unique \Star{}representation \((\pi,V)_* \colon A\rtimes_\alpha G \to \Mult(D)\) attached to \((\pi,V)\) satisfies \((\pi,V)_*(\rho_h) = (\pi,V)_*(\sigma_h)\) for all \(h\in H\).  This already implies that the universal \(\Cst\)\nb-algebra \(A\rtimes_{(\alpha,u)}(G,H,\tcm,\acm)\) must be a quotient of \(A\rtimes_\alpha G\) by a certain ideal.

The commutation relations in \(A\rtimes_\alpha G\) imply that both pairs \((\sigma, i_G)\) and \((\rho, i_G)\) define unitary representations of the crossed product group \(H\rtimes_\acm G\) in \(\Mult(A\rtimes_\alpha G)\), and the relations \(\rho_h i_A(a) = i_A\bigl(\alpha_{\tcm(h)}(a)\bigr) \rho_h\) and \(\sigma_h i_A(a) = i_A\bigl(\alpha_{\tcm(h)}(a)\bigr) \sigma_h\) for all \(h\in H\) and \(a\in A\) imply that the triples \((i_A,\rho,i_G)\) and \((i_A,\sigma,i_G)\) are covariant representations of \((A,H\rtimes_\acm G,\bar\alpha)\) in \(\Mult(A\rtimes_\alpha G)\).  These integrate to \Star{}homomorphisms
\[
\rho_*,\sigma_*\colon \Cont_c(H\rtimes_\acm G,A) \rightrightarrows \Cont_c(G,A)
\]
(see~\eqref{eq:rho_sigma_integrated}) and hence map \(A\rtimes_{\bar\alpha}(H\rtimes_\acm G)\) into \(A\rtimes_\alpha G\).  This yields \Star{}homomorphisms
\begin{equation}
  \label{eq:coequaliser_ingredients}
  \rho_*,\sigma_*\colon
  A\rtimes_{\bar\alpha} (H\rtimes_\acm G) \rightrightarrows
  A\rtimes_\alpha G.
\end{equation}

If \(f\colon A\rtimes_\alpha G\to \Mult(D)\) is the integrated form of a covariant representation \((\pi,V)\) of \((G,H,\tcm,A,\alpha,u)\), then, clearly, \(f\circ\rho_* = f\circ\sigma_*\).  The converse also holds by the universal property of the crossed product \(A\rtimes_{\bar\alpha} (H\rtimes_\acm G)\).  Thus the crossed product for \((G,H,\tcm,A,\alpha,u)\) is the coequaliser of the two maps in~\eqref{eq:coequaliser_ingredients}, that is, the quotient of \(A\rtimes_\alpha G\) by the closed two-sided ideal~\(I_H\) generated by the range of the linear map \(\rho_*-\sigma_*\colon A\rtimes_{\bar\alpha} (H\rtimes_\acm G) \to A\rtimes_\alpha G\).

Now we return to the general case of a crossed module of groupoids.  We define \Star{}homomorphisms \(\rho_*\) and~\(\sigma_*\) from \(\Cont_c(H\rtimes_\acm G,A)\) to \(\Cont_c(G,A)\) by the formulas
\begin{equation}
  \label{eq:rho_sigma_integrated}
  \begin{aligned}
    (\rho_* f)(g) &\defeq
    \int_{H_x} f(g \tcm(h)^{-1},h)\,\dd\lambda_{H_x}(h),\\
    (\sigma_* f)(g) &\defeq \int_{H_x} f(g,h)\cdot u_h \,\dd\lambda_{H_x}(h).
  \end{aligned}
\end{equation}
In the group case, these are the maps we get from the constructions above.  In general, \eqref{eq:rho_sigma_integrated} defines contractive \Star{}homomorphisms from \(\Cont_c(H\rtimes_\acm G,A)\) to \(\Cont_c(G,A)\).  These extend to \Star{}homomorphisms as in~\eqref{eq:coequaliser_ingredients}.

\begin{definition}
  \label{def:crossed_product_groupoid}
  We let \(A\rtimes_{\alpha,u} (G,H)\) be the coequaliser of the pair of maps in~\eqref{eq:coequaliser_ingredients}, that is, the quotient of \(A\rtimes_\alpha G\) by the closed ideal generated by the range of \(\rho_*-\sigma_*\).
\end{definition}

By construction, this definition extends the crossed product for ordinary locally compact groupoids, that is, it agrees with \(A\rtimes_\alpha G\) if \(H=\{1\}\).

The crossed product comes equipped with canonical \Star{}representations
\[
i_A\colon A\to \Mult\bigl(A\rtimes_{\alpha,u} (G,H) \bigr),\qquad i_G\colon \Cst(G) \defeq \Cont_0(X) \rtimes G \to \Mult\bigl(A\rtimes_{\alpha,u} (G,H) \bigr).
\]
The map~\(i_G\) comes from the map \(\Cont_0(X) \to \Mult(A)\), which induces a \Star{}homomorphism \(\Cst(G) \defeq \Cont_0(X)\rtimes G \subseteq \Mult(A)\rtimes G\subseteq \Mult(A\rtimes G)\).  In the group case, \(i_G\) is equivalent to a unitary representation of~\(G\), and \(A\rtimes_{\alpha,u} (G,H)\) with the maps \(i_A\) and~\(i_G\) satisfies the universal property required of a crossed product.

Actually, the range of \(\rho_*-\sigma_*\) is itself a two-sided ideal because both \(\rho_*\) and~\(\sigma_*\) are bimodule homomorphisms with respect to the obvious \(A\rtimes_\alpha G\)-bimodule structures on \(A\rtimes_{\bar\alpha} (H\rtimes_\acm G)\) and~\(A\rtimes_\alpha G\) (use the embedding \(A\rtimes_\alpha G\to \Mult(A\rtimes_{\bar\alpha} (H\rtimes_\acm G)\)).  Thus~\(I_H\) is the closure of the range of the map \(\rho_*-\sigma_*\).

\begin{example}
  \label{exa:crossed_product_Green}
  Let~\(G\) be a locally compact group and let~\(H\) be a closed normal subgroup in~\(G\).  Then \(A\rtimes_{\alpha,u}(G,H)\) coincides with the twisted covariance algebra in~\cite{Green:Local_twisted}.  Both satisfy the same universal property, and the explicit constructions are also equivalent (we have expanded the construction of the ideal~\(I_H\) to prepare for the generalisation of the theory to groupoids).
\end{example}

\begin{definition}
  \label{def:crossed_module_Cstar}
  Let \(A\defeq \Cont_0(X)\) with~\(G\) acting by the usual action~\(\alpha\) on its object space and \(u_h=1\) for all \(h\in H\).  This defines an action \((\alpha,1)\) of \((G,H,\tcm,\acm)\) on~\(A\).  The resulting crossed product \(\Cst\)\nb-algebra \(\Cst(G,H,\tcm,\acm)\defeq \Cont_0(X)\rtimes_{\alpha,1} (G,H)\) is the \emph{crossed module \(\Cst\)\nb-algebra} of \((G,H,\tcm,\acm)\).
\end{definition}

If \(H=\{1\}\), this definition agrees with the groupoid \(\Cst\)\nb-algebra of~\(G\).

\begin{example}
  \label{exa:crossed_module_Cstar}
  Suppose \(u_h=1\) for all \(h\in H\).  (This is impossible unless \(\alpha_{\tcm(h)}=\Id\) for all \(h\in H\).)  The covariance condition \(V_{\tcm(h)} = \pi(u_h)\) means that~\(V\) is trivial on \(\tcm(H)\).  If \(G\) and~\(H\) are groups, we may replace \(\tcm(H)\) by its closure~\(\cl{\tcm(H)}\) because \(\Aut(A)\) is Hausdorff.  Hence covariant representations of \((G,H,\tcm,A,\alpha,1)\) are the same as covariant representations of \((G/\cl{\tcm(H)},A,\alpha)\), so that
  \[
  A\rtimes_{(\alpha,1)}(G,H)\cong A\rtimes_\alpha \bigl( G/\cl{\tcm(H)} \bigr) .
  \]

  In particular, the crossed module \(\Cst\)\nb-algebra \(\C\rtimes_{\Id,1}(G,H)\) for a crossed module \((G,H)\) is \(\Cst\bigl(G/\cl{\tcm(H)}\bigr)\).  This is just~\(\C\) if~\(\tcm\) has dense range.  For instance, this shows that \(\Cst(\Torus,\Z) = \C\).  It is rather surprising that this answer should be so trivial.

  In the groupoid case, we cannot always pass to a Hausdorff quotient.  For each \(x\in X\), let~\(\bar{H}_x\) be the closure of~\(\tcm(H_x)\) in the stabiliser~\(G_x^x\) and let \(\bar{H}=\bigcup_{x\in X} \bar{H}_x\).  The argument above shows that~\(\bar{H}\) acts trivially in any covariant representation, so that our action descends to~\(G/\bar{H}\).  But~\(\bar{H}\) need not be closed in~\(G\).  For instance, in the situation of Example~\ref{exa:crossed_module_groupoid_isotropy}, \(\bar{H}=\tcm(H)\) because the fibres~\(G_x^x\) are all discrete, and this subset is always open in~\(G\) but usually not closed.  The \(\Cst\)\nb-algebra \(\Cst(G,H)\) for the crossed module in Example~\ref{exa:crossed_module_groupoid_isotropy} is isomorphic to the \(\Cst\)\nb-algebra of the non-Hausdorff groupoid \(G/H\) as described in~\cite{Connes:Survey_foliations}: the coequaliser of \(\rho_*,\sigma_*\colon \Contc(H\rtimes_\acm G) \rightrightarrows \Contc(G)\) is exactly the convolution algebra \(\Contc(G/H)\) used in~\cite{Connes:Survey_foliations}; the latter is spanned by continuous functions with compact support on Hausdorff open subsets of~\(G/H\).
\end{example}

\begin{example}
  \label{exa:crossed_product_B_group}
  Let~\(H\) be an abelian locally compact group.  Recall that actions of the crossed module \((\{1\},H)\) correspond to \(\Cst\)\nb-algebras over~\(\hat{H}\) (Example~\ref{exa:action_B_group}).  When we translate the crossed product with \((\{1\},H)\) to \(\Cst\)\nb-algebras over~\(\hat{H}\), it becomes the functor that takes the fibre at~\(1\) of a \(\Cst\)\nb-algebra over~\(\hat{H}\).  This is the quotient of~\(A\) by the relations \(u_h=1\) for all \(h\in H\).
\end{example}

Similarly, for the crossed module \((G,H,1,\acm)\) of an abelian group extension introduced in Example~\ref{exa:action_Ab_extension}, we may view actions of \((G,H,1,\acm)\) as \(G\)\nb-equivariant \(\Cst\)\nb-algebras over the space~\(\hat{H}\).  In this setting, the crossed product becomes the following construction: first restrict the bundle to the fibre of \(1\in\hat{H}\), which still carries an action of~\(G\), and then take the crossed product with~\(G\).

Recall that the correspondence between crossed module actions of \((G,H,1,\acm)\) and actions of the groupoid \(\hat{H}\rtimes G\) is an equivalence of categories.  Nevertheless, these two categories differ in important additional structure, namely, in the crossed product functors and the natural tensor products on both categories.

\begin{theorem}
  \label{the:split_crossed_product_I}
  Let \((G,H,\tcm,\acm)\) be a crossed module of locally compact topological groupoids and let~\(A\) be a \(\Cst\)\nb-algebra with a continuous action~\(\alpha\) of~\(G\).  Equip \(A\rtimes_{\alpha\circ \partial} \nobreak H\) with the canonical action of~\((G,H,\tcm,\acm)\) in Example~\textup{\ref{exa:action_on_crossed}}.  Then there is a natural isomorphism
  \[
  (A\rtimes_{\alpha\circ \partial} H)\rtimes(G,H)\cong A\rtimes_\alpha G.
  \]
\end{theorem}

\begin{proof}
  The action of~\(G\) on \(A\rtimes H\) is defined so that \((A\rtimes H)\rtimes G \cong A\rtimes (H\rtimes_\acm G)\).  The groupoid homomorphism \(\chi\colon H\rtimes_\acm G \to G\), \((h,g)\mapsto \tcm(h)g\), induces a surjective \Star{}homomorphism \(\chi_*\colon A\rtimes (H\rtimes_\acm G) \to A\rtimes G\).  On dense subalgebras, it is given by
  \[
  \chi_*\colon \Contc\bigl(G,\Contc(H,A)\bigr) \to \Contc(G,A),\qquad \chi_* f(g) \defeq \int_{H_{\source(g)}} f(g\tcm(h)^{-1},h) \,\dd\lambda_{H_{\source(g)}}(h).
  \]
  We claim that~\(\chi_*\) descends to a \Star{}homomorphism \(\dot\chi_*\colon A\rtimes H\rtimes (G,H) \to A\rtimes G\), that is, \(\chi_*\circ\rho_*= \chi_*\circ\sigma_*\), where \(\rho_*,\sigma_*\colon A\rtimes H\rtimes (H\rtimes_\acm G) \rightrightarrows A\rtimes H\rtimes G\) are the maps defined in~\eqref{eq:coequaliser_ingredients}.

  Let \(H\rtimes_{\acm\circ\chi} (H\rtimes_\acm G)\) be the crossed product groupoid over~\(X\) for the action
  \[
  \acm\circ\chi\colon H\rtimes_\acm G\to G\to\Aut(H)
  \]
  of~\(H\rtimes_\acm G\) on~\(H\).  This is \(H\times_X H\times_X G\) with the product
  \[
  (h_1,h_2,g_2)\cdot (h_3,h_4,g_4) \defeq (h_1 \cdot \acm_{\tcm(h_2)g_2}(h_3),h_2\cdot \acm_{g_2}(h_4), g_2\cdot g_4).
  \]
  The map \(\chi_2\colon H\rtimes_{\acm\chi} H\rtimes_\acm G \to G\), \((h_1,h_2,g)\mapsto \tcm(h_1h_2)g\), is a groupoid homomorphism over the the base map \(\Id\colon X \to X\).  The obvious map identifies the iterated crossed product \(A\rtimes H\rtimes (H\rtimes_\acm G)\) with \(A\rtimes_{\alpha\chi_2} (H\rtimes_{\acm\chi} H\rtimes_\acm G)\).  A routine computation shows that the two maps
  \[
  \chi_*\rho_*,\chi_*\sigma_* \colon A\rtimes_{\alpha \chi_2} (H\rtimes_{\acm\chi} H \rtimes_\acm G) \rightrightarrows A\rtimes_\alpha G
  \]
  are both induced by the groupoid homomorphism~\(\chi_2\).  Hence~\(\dot\chi_*\) exists.

  The map \(\dot\chi_*\colon (A\rtimes H)\rtimes(G,H)\to A\rtimes G\) is surjective because~\(\chi_*\) is surjective.  To establish injectivity, we must show that any representation of \(A\rtimes H\rtimes (G,H)\) factors through \(A\rtimes G\).  This is true already on the level of the dense subalgebras \(\Contc\bigl(G,\Contc(H,A)\bigr) = \Contc(G\ltimes_\acm H,A)\) in \(A\rtimes H\rtimes G\) and \(\Contc(G,A)\) in \(A\rtimes G\).  A representation of \(\Contc\bigl(G,\Contc(H,A)\bigr)\) with \(f\circ\rho_*=f\circ\sigma_*\) for \(\rho_*\) and~\(\sigma_*\) as in~\eqref{eq:rho_sigma_integrated} must factor through~\(\chi_*\).  Hence any representation of \(A\rtimes H\rtimes (G,H)\) factors through~\(\dot\chi_*\).
\end{proof}

\begin{example}
  \label{exa:crossed_group_on_itself}
  Let \((G,H,\tcm,\acm)\) be a crossed module of locally compact groups.  Let \(A=\Cont_0(G)\rtimes H\) with the canonical covariant representation of \((G,H,\tcm,\acm)\).  We may interpret~\(A\) as the \(\Cst\)\nb-algebra of functions on the underlying non-commutative space of~\((G,H,\tcm,\acm)\).  Then \(A\rtimes(G,H) \cong \Cont_0(G)\rtimes G \cong \Comp(L^2G)\).  In particular, this implies
  \[
  \Rotc\vartheta\rtimes(\Torus,\Z) \cong \Comp(L^2\Torus).
  \]
  for the action of \((\Torus,\Z)\) on the rotation \(\Cst\)\nb-algebra~\(\Rotc\vartheta\) in Example~\ref{exa:action_Torus_Z}.
\end{example}

\begin{corollary}
  \label{cor:free_covariant_rep_unique}
  Let \((G,H,\tcm,\acm)\) be a crossed module of locally compact groups.  Equip \(A\defeq \Cont_0(G)\rtimes H\) with the canonical action of \((G,H,\tcm,\acm)\).  Any covariant representation of \((G,H,\tcm,A,\alpha,u)\) is a direct sum of multiples of the standard covariant representation on \(L^2G\), which involves the actions of \(\Cont_0(G)\) by pointwise multiplication and the actions of \(H\) and~\(G\) by left translations.
\end{corollary}

\begin{proof}
  By the universal property of \(A\rtimes_{\alpha,u} (G,H)\), covariant representations of \((G,H,\tcm,A,\alpha,u)\) correspond to representations of \(A\rtimes_{\alpha,u} (G,H)\).  Theorem~\ref{the:split_crossed_product_I} identifies this crossed product \(\Cst\)\nb-algebra with \(\Cont_0(G)\rtimes G\), which is well-known to be isomorphic to \(\Comp(L^2G)\) because the transformation groupoid \(G\rtimes G\) is isomorphic to the pair groupoid on~\(G\).  Any \Star{}representation of \(\Comp(L^2G)\) is a direct sum of copies of the standard representation on~\(L^2G\).
\end{proof}

\section{Functoriality and exactness of crossed products}
\label{sec:functor_exact}

Throughout this section, we let \(\cm= (G,H,\tcm,\acm)\) be a crossed module of locally compact groupoids with Haar systems, so that crossed products with~\(\cm\) are defined.  Let~\(X\) be the common object space of \(G\) and~\(H\).

\begin{definition}
  \label{def:equivariant_map}
  Let \((\alpha,u)\) and \((\beta,v)\) be actions of a crossed module \((G,H,\tcm,\acm)\) on \(\Cst\)\nb-algebras \(A\) and~\(B\), respectively.  We call a \Star{}homomorphism \(\pi\colon A\to \Mult(B)\) \emph{\((G,H)\)-equivariant} if \(\pi\bigl(\alpha_g(a)\bigr)=\beta_g\bigl(\pi(a)\bigr)\) and \(\pi(u_ha)=v_h\pi(a)\) for all \(a\in A_x\), \(g\in G\) with \(\source(g)=x\), and all \(h\in H_x\).
\end{definition}

The condition \(\pi(u_ha)=v_h\pi(a)\) generalises \(\pi(u_h)=v_h\) to possibly degenerate \Star{}homomorphisms; the latter condition involves the strictly continuous extension of~\(\pi\), which only exists in the non-degenerate case.

The crossed product \(A\rtimes (G,H)\) described above is clearly functorial for \((G,H)\)-equivariant maps: a \((G,H)\)\nb-equivariant map \(f\colon A\to B\) induces commuting diagrams of \Star{}homomorphisms
\[
\xymatrix{ A\rtimes_{\bar\alpha} (H\rtimes_\acm G) \ar[d]^{f_*} \ar[r]^-{\rho_*}&
  A\rtimes_\alpha G \ar[d]^{f_*}\\
  B\rtimes_{\bar\beta} (H\rtimes_\acm G) \ar[r]^-{\rho_*}& B\rtimes_\beta G }\qquad\xymatrix{ A\rtimes_{\bar\alpha} (H\rtimes_\acm G) \ar[d]^{f_*} \ar[r]^-{\sigma_*}&
  A\rtimes_\alpha G \ar[d]^{f_*}\\
  B\rtimes_{\bar\beta} (H\rtimes_\acm G) \ar[r]^-{\sigma_*}& B\rtimes_\beta G.  }
\]
This induces a map \(f\rtimes(G,H)\colon A\rtimes (G,H) \to B\rtimes (G,H)\) between the coequalisers.  Similarly, a non-degenerate \((G,H)\)-equivariant \Star{}homomorphism \(A\to \Mult(B)\) induces a non-degenerate \Star{}homomorphism \(A\rtimes (G,H)\to \Mult\bigl(B\rtimes (G,H)\bigr)\).

Next we claim that the crossed product functor preserves Morita equivalence, provided we use the appropriate notion of equivariant Morita equivalence for actions of crossed modules.  The most systematic way to arrive at this notion uses the linking algebra picture of Morita equivalence.  Recall that two \(\Cst\)\nb-algebras \(A\) and~\(B\) are Morita equivalent if and only if there are another \(\Cst\)\nb-algebra~\(D\) and a projection \(p\in\Mult(D)\) such that both \(p\) and~\(p^\bot\defeq 1-p\) are full (\(DpD\) and \(Dp^\bot D\) span a dense subspace of~\(D\)) and \(A\cong pDp\), \(B\cong p^\bot D p^\bot\).

Given an imprimitivity \(A,B\)\nb-bimodule~\(\E\), we may let
\[
D\defeq
\begin{pmatrix}
  A&\E\\\E^*&B
\end{pmatrix},\qquad p\defeq
\begin{pmatrix}
  1&0\\0&0
\end{pmatrix}.
\]
The multiplication in~\(D\) is matrix multiplication together with the maps \(\E\times\E^*\to\nobreak A\), and so on, which are part of the imprimitivity bimodule structure.  Conversely, given \(D\) and~\(p\), we let \(\E\defeq pDp^\bot\) with the canonical Hilbert \(A,B\)-bimodule structure.

\begin{definition}
  \label{def:Morita_equivalence_equiv}
  Two \(\Cst\)\nb-algebra actions \((A,\alpha,u)\) and \((B,\beta,v)\) of~\(\cm\) are called (equivariantly) \emph{Morita equivalent} if there is another \(\Cst\)\nb-algebra action \((D,\delta,w)\) of~\(\cm\) and a \(\delta\)\nb-invariant projection \(p\in\Mult(D)\) such that
  \begin{itemize}
  \item both \(p\) and~\(p^\bot\defeq 1-p\) are full,
  \item there are \((G,H)\)\nb-equivariant \Star{}isomorphisms \(pDp\cong A\) and \(p^\bot Dp^\bot\cong B\).
  \end{itemize}
  Since \(p\) is \(\delta\)\nb-invariant, the action of~\(G\) on~\(D\) restricts to one on \(pDp\subseteq D\) and \(w_hp_x w_h^* = \delta_{\tcm(h)}(p_x) = p_x\) for all \(h\in H_x\), that is, \(w\) commutes with~\(p\).  Therefore, the unitary multipliers~\(w_h\) restrict to unitary multipliers \(pw_hp\) and \(p^\bot w_h p^\bot\) of the appropriate fibres of \(pDp\) and \(p^\bot Dp^\bot\), respectively.  These together with the restrictions of~\(\delta\) form actions of \((G,H)\) on the corners \(pDp\) and \(p^\bot Dp^\bot\), which are used above.
\end{definition}

\begin{proposition}
  \label{pro:crossed_Morita}
  Let \((A,\alpha,u)\) and \((B,\beta,v)\) be equivariantly Morita equivalent \(\Cst\)\nb-algebra actions of~\(\cm\).  Then the crossed products \(A\rtimes_{\alpha,u} \cm\) and \(B\rtimes_{\beta,v} \cm\) are Morita equivalent as well.
\end{proposition}

\begin{proof}
  Let \((D,\delta,w)\) and \(p\in\Mult(D)\) implement the Morita equivalence between \(A\) and~\(B\).  Using the canonical strictly continuous unital embeddings of~\(\Mult(D)\) into \(\Mult(D\rtimes_\delta G)\) and \(\Mult\bigl(D\rtimes_{\bar\delta} (H\rtimes_\acm G)\bigr)\), we map \(p\) and~\(p^\bot\) to complementary projections in these crossed products.  The Morita equivalence of crossed products for ordinary groupoid actions amounts to the statement that these projections are again full and that the resulting corners are \(A\rtimes_\alpha G\), \(B\rtimes_\beta G\), \(A\rtimes_{\bar\alpha} (H\rtimes_\acm G)\), and \(B\rtimes_{\bar\beta} (H\rtimes_\acm G)\), respectively.  Passing now to the crossed products with \((G,H)\), we conclude that \(D\rtimes_{\delta,w} (G,H)\) and the image of~\(p\) in \(\Mult\bigl(D\rtimes_{\delta,w} (G,H)\bigr)\) provide a Morita equivalence between \(A\rtimes_\alpha (G,H)\) and \(B\rtimes_\beta (G,H)\).
\end{proof}

More explicitly, we may also describe Morita equivalence of crossed module actions by imprimitivity bimodules.  If \((D,\delta,w)\) and \(p\in\Mult(D)\) provide a Morita equivalence between \((A,\alpha,u)\) and \((B,\beta,v)\), then we may interpret \(\E\defeq pDp^\bot\) as an \(A,B\)-imprimitivity bimodule.  Like~\(D\), \(\E\) is a continuous field of Banach spaces over~\(X\).  Since~\(p\) is \(\delta\)\nb-invariant, the action of~\(G\) on~\(D\) restricts to an action~\(\gamma\) on~\(\E\).  This action satisfies the compatibility conditions
\begin{equation}
  \label{eq:action_Hilbert_module}
  \begin{alignedat}{2}
    \gamma_g(a\cdot\xi)&= \alpha_g(a)\cdot\gamma_g(\xi),&\qquad \alpha_g(\braket{\xi}{\eta}_A)&=
    \braket{\gamma_g(\xi)}{\gamma_g(\eta)}_A\\
    \gamma_g(\xi\cdot b)&= \gamma_g(\xi)\cdot\beta_g(b),&\qquad \beta_g(\braket{\xi}{\eta}_B)&= \braket{\gamma_g(\xi)}{\gamma_g(\eta)}_B,
  \end{alignedat}
\end{equation}
for all \(g\in G\) and \(a\), \(\xi\), \(\eta\), and~\(b\) in appropriate fibres of \(A\), \(\E\), and~\(B\).  Furthermore, the multipliers~\(w_h\) for \(h\in H\) restrict to multipliers of the fibres of~\(\E\).  But we do not have to specify these as additional data because we may write any element of~\(\E\) as a product \(a\cdot\xi\cdot b\), and~\(w_h\) is totally determined by \(u_h\) and~\(v_h\) in the sense that
\[
w_h\cdot (a\cdot\xi\cdot b) = (u_h\cdot a)\cdot \xi\cdot b,\qquad (a\cdot\xi\cdot b)\cdot w_h = a\cdot \xi\cdot (b\cdot v_h).
\]
This becomes \(w_h\cdot \xi=u_h\cdot\xi\) and \(\xi\cdot w_h = \xi\cdot v_h\) when we extend the \(A,B\)-bimodule structure on~\(\E\) to \(\Mult(A)\) and \(\Mult(B)\).  The action of~\(G\) must also satisfy
\begin{equation}
  \label{eq:triviality_condition_imprimitivity_bimodule}
  \gamma_{\tcm(h)}(\xi) = w_h\cdot \xi\cdot w_h^*
  = u_h\cdot \xi\cdot v_h^*
\end{equation}
for all \(h\in H_x\), \(\xi\in \E_x\), \(x\in X\).

Conversely, if we are given an \(A,B\)\nb-imprimitivity bimodule~\(\E\) with an action of~\(G\) that satisfies \eqref{eq:action_Hilbert_module} and~\eqref{eq:triviality_condition_imprimitivity_bimodule}, then the resulting linking algebra~\(D\) carries an action \((\delta,w)\) of \((G,H)\) for which the projection~\(p\) is \(\delta\)\nb-invariant.  Summing up:

\begin{proposition}
  \label{pro:Morita_via_bimodule}
  Two actions \((A,\alpha,u)\) and \((B,\beta,v)\) of~\(\cm\) are Morita equivalent if and only if there is an \(A,B\)-imprimitivity bimodule~\(\E\) with an action of~\(G\) that satisfies \eqref{eq:action_Hilbert_module} and~\eqref{eq:triviality_condition_imprimitivity_bimodule}.
\end{proposition}

Dropping the conditions involving \((A,\alpha,u)\) in the above definitions and the fullness of the \(B\)\nb-valued inner product, we get a \emph{\((G,H)\)\nb-equivariant right Hilbert \(B\)\nb-module}.  This is a Hilbert \(B\)\nb-module with an action~\(\gamma\) of~\(G\) satisfying the second half of~\eqref{eq:action_Hilbert_module}.  Equation~\eqref{eq:triviality_condition_imprimitivity_bimodule} now becomes a definition: \(w_h\cdot \xi \defeq \gamma_{\tcm(h)}(\xi)\cdot v_h\) for all \(h\in H_x\), \(\xi\in\E_x\).  It is easy to check that this defines a unitary, hence adjointable, operator on~\(\E_x\) and that \(h\mapsto w_h\) defines a continuous homomorphism from~\(H\) to the group bundle of unitary multipliers of the fibres~\(\E_x\) of~\(\E\).  The induced action on \(\Comp(\E)\) and the unitaries~\(w_h\) define an action of~\(\cm\) on \(\Comp(\E)\).  If~\(\E\) is full for the \(B\)\nb-valued inner product, then~\(\E\) becomes a \((G,H)\)-equivariant \((\Comp(\E),B)\)-imprimitivity bimodule.

We may also turn \(\Cst\)\nb-algebras into a category using isomorphism classes of correspondences as arrows instead of \Star{}homomorphisms.  The equivariant analogue of a correspondence is defined as follows: an \emph{equivariant correspondence} from \((A,\alpha,u)\) to \((B,\beta,v)\) is a \((G,H)\)-equivariant non-degenerate \Star{}homomorphism from \((A,\alpha,u)\) to the algebra of adjointable operators on an equivariant Hilbert module over \((B,\beta,v)\).  It is easy to see that an equivariant correspondence from \((A,\alpha,u)\) to \((B,\beta,v)\) induces a correspondence from \(A\rtimes_{\alpha,u} (G,H)\) to \(B\rtimes_{\beta,v} (G,H)\).

We finish our discussion of formal properties of the crossed product with a discussion of its exactness.  First we must define extensions of \(\Cst\)\nb-algebra actions of~\(\cm\).  Let~\((A,\alpha,u)\) be a \(\Cst\)\nb-algebra with an action of~\(\cm\) and let \(I\triangleleft A\) be an ideal in~\(A\) that is invariant under the action of~\(G\).  Then~\(\alpha\) induces actions \(\alpha_I\) and~\(\alpha_{A/I}\) on \(I\) and~\(A/I\).  Any multiplier of~\(A\) restricts to a multiplier of~\(I\), that is, \(I\) is still an ideal in \(\Mult(A)\).  Thus multipliers of~\(A\) descend to multipliers of~\(A/I\).  Thus~\(u\) also yields fibrewise unitary multipliers of \(I\) and~\(A/I\).  This produces actions \((\alpha_I,u_I)\) and \((\alpha_{A/I},u_{A/I})\) of \((G,H,\tcm,\acm)\) on \(I\) and~\(A/I\), respectively.

\begin{definition}
  \label{def:extension_actions}
  A diagram \(K\to E\to Q\) of \(\Cst\)\nb-algebra actions of~\(\cm\) is called an \emph{extension} if it is isomorphic to a diagram of the form
  \[
  (I,\alpha_I,u_I) \to (A,\alpha,u) \to (A/I,\alpha_{A/I},u_{A/I})
  \]
  for some \(\Cst\)\nb-algebra action~\((A,\alpha,u)\) and some \(G\)\nb-invariant ideal~\(I\).
\end{definition}

\begin{proposition}
  \label{pro:crossed_exact}
  The crossed product functor \((A,\alpha,u)\mapsto A\rtimes_{\alpha,u} \cm\) is exact, that is, it maps extensions of \(\Cst\)\nb-algebra actions of~\(\cm\) to extensions of \(\Cst\)\nb-algebras.
\end{proposition}

\begin{proof}
  It is well-known that the full crossed product functor for groupoid crossed products is exact.  Therefore, we get exact sequences
  \begin{equation}
    \label{eq:exact_crossed_groupoid}
    \begin{gathered}
      I\rtimes_{\alpha_I} G \into A\rtimes_{\alpha} G \onto
      A/I\rtimes_{\alpha_{A/I}} G,\\
      I\rtimes_{\bar\alpha_I} (H\rtimes_\acm G) \into A\rtimes_{\bar\alpha} (H\rtimes_\acm G) \onto A/I\rtimes_{\bar\alpha_{A/I}} (H\rtimes_\acm G).
    \end{gathered}
  \end{equation}
  It follows immediately that the map \(A\rtimes (G,H) \to (A/I)\rtimes (G,H)\) is surjective.

  Next we show that the map \(\pi\colon I\rtimes (G,H) \to A\rtimes (G,H)\) is injective, by showing that any representation of \(I\rtimes (G,H)\) extends to one of \(A\rtimes (G,H)\).

  A representation of \(I\rtimes (G,H)\) is nothing but a representation~\(f\) of \(I\rtimes G\) with \(f\circ\sigma_* = f\circ \rho_*\), where \(\sigma_*\) and~\(\rho_*\) are the maps in \eqref{eq:coequaliser_ingredients} and~\eqref{eq:rho_sigma_integrated}.  The representation~\(f\) extends uniquely to a representation~\(\bar{f}\) of \(A\rtimes G\) because \(I\rtimes G\) is an ideal in \(A\rtimes G\) (recall that representation means non-degenerate \Star{}representation throughout this article).  We have \(\bar{f}\circ\sigma_* = \bar{f}\circ\rho_*\) on \(A\rtimes (H\rtimes_\acm G)\) because both sides extend the same representation of \(I\rtimes (H\rtimes_\acm G)\).  Hence~\(\bar{f}\) descends to a representation of \(A\rtimes (G,H)\).  As a result, any representation of \(I\rtimes (G,H)\) extends uniquely to a representation of \(A\rtimes (G,H)\).  This shows that the map \(I\rtimes (G,H) \to A\rtimes (G,H)\) is injective.

  Finally, we verify exactness in the middle by checking that a representation of \(A\rtimes (G,H)\) that vanishes on \(I\rtimes (G,H)\) descends to a representation of \(A/I\rtimes (G,H)\).  We may view the representation of \(A\rtimes (G,H)\) as a representation~\(f\) of \(A\rtimes G\) with \(f\circ\rho_* = f\circ\sigma_*\).  The vanishing on \(I\rtimes (G,H)\) is equivalent to~\(f\) vanishing on \(I\rtimes G\) because \(I\rtimes (G,H)\) is a quotient of \(I\rtimes G\).  Hence~\(f\) descends to a representation~\(\dot{f}\) of \(A/I\rtimes G\) by the exactness of~\eqref{eq:exact_crossed_groupoid}.  The latter satisfies \(\dot{f}\circ\rho_* = \dot{f}\circ\sigma_*\) on \((A/I)\rtimes (H\rtimes_\acm G)\) because \(f\circ\rho_* = f\circ\sigma_*\) holds on \(A\rtimes (H\rtimes_\acm G)\).  Hence~\(\dot{f}\) descends to a representation of \((A/I)\rtimes (H\rtimes_\acm G)\), as desired.
\end{proof}

\section{An example}
\label{sec:example}

In this section, we consider actions on \(\Cst\)\nb-algebras and crossed products for a special case of the construction in~\eqref{exa:crossed_module_groupoid_isotropy}.  We also describe a candidate for a universal proper action in this case.

A \(\Z\)\nb-action is determined by the action of its generator \(1\in\Z\), and~\(1\) acts by a homeomorphism~\(\varphi\). Hence a \(\Z\)\nb-action is equivalent to a single homeomorphism~\(\varphi\).  We take our space to be~\(\R\) and
\[
\varphi\colon \R\to \R,\qquad x\mapsto
\begin{cases}
  x&\text{if \(x\le0\),}\\
  x/2&\text{if \(x\ge0\).}
\end{cases}
\]
The precise formula for the homeomorphism~\(\varphi\) is not important. We will only use that \(\varphi(x)=x\) for \(x\le 0\).  We consider the transformation groupoid \(G=(\R\rtimes_\varphi \Z \rightrightarrows \R)\), where~\(\Z\) acts by \(n \cdot x = \varphi^{(n)} (x)\).

The interior isotropy group bundle~\(H\) consists of arrows whose germs act identically on some neighbourhood.  Here \(H=\R_{<0}\times\Z \cup \R_{\le0}\times\{0\}\), that is, the groups~\(H_x\) are~\(\Z\) for \(x<0\) and \(\{0\}\) for \(x\ge0\).  Thus the locally Hausdorff quotient groupoid~\(G/H\) is indeed non-Hausdorff at \(\{0\}\times\Z\), and Hausdorff over \(\R_{>0}\) and~\(\R_{<0}\).  This groupoid is already considered by Alain Connes in~\cite{Connes:Survey_foliations}*{\S6} to illustrate the difficulties that appear for non-Hausdorff holonomy groupoids of foliations like the Reeb foliation.

Let~\(\cm\) be the crossed module \(H\to G\).  What are continuous actions of~\(\cm\) on \(\Cst\)\nb-algebras?  First, the underlying \(\Cst\)\nb-algebra must be a \(\Cst\)\nb-algebra~\(A\) over~\(\R\); let \(A(U)\) for an open subset \(U\subseteq\R\) be the ideal in~\(A\) spanned by \(\Cont_0(U)\cdot A\).

Since~\(G\) is a transformation group for a \(\Z\)\nb-action, an action~\(\alpha\) of the groupoid~\(G\) on~\(A\) amounts to an action \(\alpha\colon \Z\to\Aut(A)\) that is compatible with the action of~\(\Z\) on~\(\R\); more precisely, this means that~\(\alpha_n\) for \(n\in\Z\) descends to maps \(A_x\to A_{\varphi^n(x)}\) for all \(x\in\R\).  Of course, the action~\(\alpha\) is determined by the single automorphism \(\alpha_1\colon A\to A\), which must map \(A_x \to A_{\varphi(x)}\) for all \(x\in\R\).

A continuous action of~\(\cm\) consists of a continuous action of~\(G\) together with a continuous representation of the group bundle \(H=\R_{<0}\times\Z\) by unitary multipliers of the fibres of~\(A\).  Equivalently, we must specify a single unitary multiplier~\(u_1\) of \(A(\R_{<0})\).  The compatibility conditions for actions of~\(\cm\) become \(\alpha_1(a) = u_1 a u_1^*\) for all \(a\in A(\R_{<0})\): the other condition \(\alpha_1(u_1)=u_1\) follows.

Now we describe the crossed product for an action \((\alpha,u)\) of~\(\cm\) on a \(\Cst\)\nb-algebra~\(A\) over~\(\R\), using the long exact sequence of Proposition~\ref{pro:crossed_exact}.  Let \(A_<\) and~\(A_\ge\) be the restrictions of~\(A\) to \(\R_{<0}\) and~\(\R_{\ge0}\), again viewed as \(\Cst\)\nb-algebras over~\(\R\).  Thus~\(A_<\) is a \(\Z\)\nb-invariant ideal in~\(A\) with quotient \(A/A_<\cong A_\ge\).  We get an induced exact sequence
\[
A_<\rtimes\cm \into A\rtimes\cm \onto A_\ge \rtimes\cm.
\]
The crossed product \(A_\ge\rtimes\cm\) is simply isomorphic to \(A_\ge\rtimes_\alpha \Z\) because~\(H\) only occurs over~\(\R_{<0}\).  For \(A_<\rtimes\cm\), only the restriction of~\(\cm\) to~\(\R_{<0}\) is relevant, and in this region \(\tcm\colon H\to G|_{\R_{<0}}\) is an isomorphism.  It follows that the map \(A_<\rtimes\Z\to A_<\) induced by the covariant representation \((\Id_{A_<},u)\) induces an isomorphism \(A_<\rtimes\cm \cong A_<\).  Thus \(A\rtimes\cm\) is an extension
\[
A_< \into A\rtimes\cm \onto A_\ge\rtimes_\alpha \Z.
\]
Up to isomorphism \(A\rtimes\cm\) is determined by the Busby invariant of this extension.  This \Star{}homomorphism \(A_\ge\rtimes_\alpha\Z \to \Mult(A_<)/A_<\) must come from a covariant pair of representations of~\(A_\ge\) and~\(\Z\).  The representation of~\(A_\ge\) is simply the Busby invariant of the extension \(A_<\into A\onto A_\ge\) because the maps \(A_? \to \Mult(A_?\rtimes\cm)\) are natural; the representation of~\(\Z\) is~\(u\).

We get the \(\Cst\)\nb-algebra \(\Cst(\cm)\) if we let~\(A\) be \(\Cont_0(\R)\) with the action of~\(\Z\) induced by~\(\varphi\).  Using our explicit description of \(A\rtimes\cm\), it is easy to check that \(\Cst(\cm)\) is the \(\Cst\)\nb-algebra of the locally Hausdorff groupoid \(G/H\) defined in~\cite{Connes:Survey_foliations}.

Finally, we describe an interesting action of~\(\cm\) on a locally compact groupoid~\(K\) that should play the role of the universal proper action for~\(\cm\).  The groupoid~\(K\) comes from an equivalence relation on the space
\[
K^{(0)} \defeq \R^3  \quad \sqcup\quad \R_{<0}\times\C,
\]
which we map to \(G^{(0)}=H^{(0)} = \R\) by projecting to the first coordinate on both components.  The equivalence relation \(K^{(1)}\subseteq K^{(0)}\times K^{(0)}\) is the equivalence relation generated by
\[
\R^3\supset \R_{<0} \times \R^2 \ni (t,x,y) \sim (t, \Eul^{y+2\pi\ima x}) \in \R_{<0}\times\C.
\]
That is, two elements \((t,x,y)\) and \((t',x',y')\) of~\(\R^3\) are equivalent if and only if they are equal or \(t=t'<0\), \(x-x'\in\Z\), and \(y=y'\); two elements of \(\R_{<0}\times\C\) are equivalent if and only if they are equal, and \((t,x,y)\in\R^3\) and \((t',z)\in\R_{<0}\times\C\) are equivalent if and only if \(t=t'\) and \(z= \Eul^{y+2\pi\ima x}\).

The space of equivalence classes of~\(K\) is a locally Hausdorff space over~\(\R\) with fibre~\(\C\) for \(t<0\) and~\(\R^2\) for \(t\ge0\); these are glued in a non-Hausdorff way at~\(0\), using the covering map \(\R^2 \to \R/\Z\times\R \cong \C^*\), \((x,y)\mapsto \Eul^{y+2\pi\ima x}\).

We let~\(\Z\) act on~\(K^{(0)}\) by
\[
\hat\varphi^{(n)}(t,x,y) \defeq (\varphi^{(n)}(t),x+n,y),\qquad
\text{for all \((t,x,y)\in\R^3\), \(n\in\Z\),}
\]
and \(\hat\varphi^{(n)}(t,z) \defeq (t,z)\) for all \((t,z)\in\R_{<0}\times\C\), \(n\in\Z\).  The observations in Example~\ref{exa:action_on_equivalence_relation} show that this defines an action of~\(\cm\) on~\(K\) because the action preserves the equivalence relation and \(\hat\varphi(k)\sim k\) for all \(k\in K^{(0)}\).

The crossed product \(\Cst(K)\rtimes\cm\) for the induced action on \(\Cst(K)\) is naturally isomorphic to the \(\Cst\)\nb-algebra of the crossed module of topological groupoids \(K\rtimes\cm = (K\rtimes\cm)_2 \to (K\rtimes\cm)_1\) --~this generalises the observation that \(\Cont_0(X)\rtimes G \cong \Cst(X\rtimes G)\) for a groupoid action on a space.  The groupoid \((K\rtimes\cm)_1\) has object space~\(K^{(0)}\).  The arrow space is
\begin{multline*}
  (K\rtimes\cm)^{(1)} = \{(k,n,k')\in K^{(0)}\times\Z\times K^{(0)} \mid k\sim \hat\varphi^{(n)}(k')\}
  \\\cong K^{(0)} \times \Z \times_{\hat\varphi^{(-)}(-), K^{(0)}, \range} K^{(1)},
\end{multline*}
with range and source maps \(\range(k,n,k') \defeq k\), \(\source(k,n,k') \defeq k'\), and the composition \((k,n,k')\cdot (k',m,k'') \defeq (k,n+m,k'')\).  This yields arrows \((t,x,y) \leftrightarrow (\varphi^{(n)}(t),x+n,y)\) for \(n\in\Z\) and \(t\ge0\), and arrows \((t,x,y)\leftrightarrow (t,x',y) \leftrightarrow (t,z)\) if \(t\in\R_{<0}\), \(x,x',y\in\R\), \(z=\Eul^{y+2\pi i x} \in\C\), \(x-x'\in\Z\).  The isotropy groups are trivial for \(t\ge0\) and~\(\Z\) for \(t<0\).  Finally, \((K\rtimes\cm)_2\) is exactly this isotropy bundle of \((K\rtimes\cm)_1\), which is a closed and open subset of the space of arrows in~\((K\rtimes\cm)_1\).

Since~\((K\rtimes\cm)_2\) is closed in the arrow space of \((K\rtimes\cm)_1\), the crossed module \(K\rtimes\cm\) is equivalent to the groupoid \(K'\defeq (K\rtimes\cm)_1/(K\rtimes\cm)_2\), that is, to the groupoid of the orbit equivalence relation for~\(K\rtimes\cm\).  We do not explain the relevant notion of equivalence here, but merely observe that \(\Cst(K\rtimes\cm)\) is canonically isomorphic to the \(\Cst\)\nb-algebra of the locally compact groupoid~\(K'\).

The groupoid~\(K'\) comes from the equivalence relation on~\(K^{(0)}\) generated by the \(\Z\)\nb-action and the equivalence relation~\(K^{(1)}\) together, that is, \(k_1 \sim' k_2\) if and only if there is \(n\in\Z\) with \(k_1\sim \hat\varphi^{(n)}(k_2)\).  This relation is closed in \(K^{(0)}\times K^{(0)}\).  Hence the groupoid~\(K'\) is free and proper, so that \(\Cst(K')\) is Morita--Rieffel equivalent to the algebra of \(\Cont_0\)\nb-functions on the space~\(K''\) of equivalence classes of~\(K'\).  It remains to describe this quotient space.

The equivalence relation on \(\R^3\subseteq K^{(0)}\) is simply \((t,x,y)\sim (\varphi^{(n)}(t),x+n,y)\) for all \((t,x,y)\in\R^3\), \(n\in\Z\).  The resulting space of equivalence relations from this piece is \(T_\varphi\times\R\), where~\(T_\varphi\) denotes the mapping torus of~\(\varphi\).  The equivalence relation on \(\R_{<0}\times\C\) is still trivial, and we glue together \(T_\varphi\times\R\) and \(\R_{<0}\times\C\) by identifying \([(t,x,y)] \sim (t,\Eul^{y+2\pi\ima x})\) for all \((t,x,y)\in\R^3\) with \(t<0\).  Equivalently, we identify the restriction \(T_\varphi|_{t<0}\times\R\) with \(\R_{<0}\times\R/\Z\times\R\cong \R_{<0}\times \C^*\) and glue \(T_\varphi\times\R\) and \(\R_{<0}\times\C\) along this common open subset.

The above computation shows that the topological \(2\)\nb-groupoid \(K\rtimes\cm\) is equivalent to an ordinary topological space~\(K''\).  This means that~\(\cm\) acts freely and properly on~\(K\).  Furthermore, the orbit space of~\(K\) is locally Hausdorff and locally quasi-compact and carries a free and proper action of the locally Hausdorff quotient groupoid~\(G/H\) in the sense of~\cite{Tu:Non-Hausdorff}.  The fibres of the orbit space of~\(K\) over~\(\R\) are either \(\C\) or~\(\R^2\) and hence contractible.  This is a necessary condition for a proper action being universal (see~\cite{Baum-Connes-Higson:BC}), which is also sufficient for torsion-free groups under some additional technical assumptions.  We have not yet studied criteria for universal proper actions of crossed modules on groupoids, but it seems plausible that the space~\(K\) will turn out to be a universal proper action of~\(\cm\).  Thus~\(K''\) plays the role of the classifying space of~\(\cm\).

\section{Summary and outlook}
\label{sec:outlook}

We have encoded non-Hausdorff symmetry groups of \(\Cst\)\nb-algebras using crossed modules.  Actions of crossed modules generalise twisted actions in the sense of Green.  We have defined crossed products for crossed module actions and computed several examples, and we checked some formal properties of such crossed products.

Crossed modules of groups and groupoids are equivalent to strict \(2\)\nb-groups and \(2\)\nb-groupoids, respectively.  We will explore this in~\cite{Buss-Meyer-Zhu:Higher_twisted} to understand notions such as Busby--Smith twisted actions, outer equivalence, and saturated Fell bundles from the point of view of higher category theory.

A long term goal of this project is to desingularise locally Hausdorff, \'etale groupoids, replacing them by Hausdorff objects with an additional layer of arrows between arrows.  This provides us with more actions of non-Hausdorff groupoids on \(\Cst\)\nb-algebras, including actions that deserve to be called proper.  This is a prerequisite for extending the Dirac dual Dirac method to non-Hausdorff groupoids.

To progress towards the Baum--Connes conjecture along these lines, we first have to replace crossed modules by more general weak \(2\)\nb-groupoids because only very special locally Hausdorff groupoids are obtained from crossed modules.  We plan to study actions of weak \(2\)\nb-groupoids and crossed products for such actions in a sequel.

Another problem largely independent from this is to construct universal proper actions of crossed modules (or weak \(2\)\nb-groupoids).  This would allow to formulate a Baum--Connes conjecture for crossed modules.  Already the case of the crossed module \(\Z\to\Torus\) that acts on rotation algebras seems interesting.

The most important general result about Green twisted actions is that such actions may be strictified by passing to a Morita equivalent \(\Cst\)\nb-algebra (see~\cite{Echterhoff:Morita_twisted}).  We plan a sequel studying generalisations of this to crossed modules.

\end{document}